\documentclass{amsproc}

\usepackage{amssymb}

\usepackage{graphicx}

\usepackage[cmtip,all]{xy}

\usepackage[all]{xy}
\usepackage{amsthm}
\usepackage{amscd}
\usepackage{amsmath}
\usepackage{amsfonts}
\usepackage{amssymb}
\usepackage{stmaryrd}
\usepackage{graphicx}%
\usepackage{hyperref}
\hypersetup{
    bookmarks=true,         
    unicode=false,          
    pdftoolbar=true,        
    pdfmenubar=true,        
    pdffitwindow=false,     
    pdfstartview={FitH},    
    pdftitle={My title},    
    pdfauthor={Author},     
    pdfsubject={Subject},   
    pdfcreator={Creator},   
    pdfproducer={Producer}, 
    pdfkeywords={keyword1, key2, key3}, 
    pdfnewwindow=true,      
    colorlinks=false,       
    linkcolor=red,          
    citecolor=green,        
    filecolor=magenta,      
    urlcolor=cyan           
}

\usepackage{amscd,keyval}
\usepackage{setspace}
\makeatletter
\define@key{modCD}{cols}{\setlength{\minCDarrowwidth}{#1}}
\define@key{modCD}{rows}{\setlength{\modCD@rowsep}{#1}}
\newlength{\modCD@rowsep}

\renewenvironment{CD}[1][]
 {\modCD@rowsep=20\ex@ 
  \setkeys{modCD}{#1}%
  \CDat
  \bgroup\relax\let\ampersand@&\iffalse}\fi
  \CD@true\vcenter\bgroup\let\\\math@cr\restore@math@cr\default@tag
  \tabskip\z@skip\baselineskip20\ex@
  &\hfill$\m@th##$\hfill\crcr}
 {\crcr\egroup\egroup\egroup}
\atdef@ V#1V#2V{\CD@check{V..V..V}{%
  \llap{$\m@th\vcenter{\hbox{$\scriptstyle#1$}}$}%
  \left\downarrow\vbox to.5\modCD@rowsep{}\right.\kern-\nulldelimiterspace
  \rlap{$\m@th\vcenter{\hbox{$\scriptstyle#2$}}$}&&}}
\atdef@ A#1A#2A{\CD@check{A..A..A}{%
  \llap{$\m@th\vcenter{\hbox{$\scriptstyle#1$}}$}%
  \left\uparrow\vbox to.5\modCD@rowsep{}\right.\kern-\nulldelimiterspace
  \rlap{$\m@th\vcenter{\hbox{$\scriptstyle#2$}}$}&&}}
\makeatother

\newtheorem{theorem}{Theorem}[section]
\newtheorem{lemma}[theorem]{Lemma}

\theoremstyle{definition}
\newtheorem{definition}[theorem]{Definition}
\newtheorem{example}[theorem]{Example}
\newtheorem{proposition}[theorem]{Proposition}

\theoremstyle{remark}
\newtheorem{remark}[theorem]{Remark}

\numberwithin{equation}{section}

\begin{document}

\title{Quasi-Elliptic Cohomology and its Power Operations}



\author{Zhen Huan}

\address{Zhen Huan, Department of Mathematics,
Sun Yat-sen University, Guangzhou, 510275 China} \curraddr{} \email{huanzhen@mail.sysu.edu.cn}
\thanks{The author was partially supported by NSF grant DMS-1406121.}


\subjclass[2010]{Primary 55}

\date{September 20, 2016}

\begin{abstract}
Quasi-elliptic cohomology is a variant of
Tate K-theory. It is the orbifold K-theory of a space of constant 
loops. For global quotient orbifolds, it can be expressed in terms
of equivariant K-theories. In this paper we show how this theory
is equipped with power operations. We also prove that the Tate
K-theory of symmetric groups modulo a certain transfer ideal
classify the finite subgroups of the Tate curve.
\end{abstract}

\maketitle 

\section{Introduction}

An elliptic cohomology theory is an even periodic multiplicative
generalized cohomology theory whose associated formal group is the
formal completion of an elliptic curve. It is an old idea of
Witten, as shown in \cite{LProceeding}, that the elliptic
cohomology of  a space $X$ is related to the
$\mathbb{T}-$equivariant K-theory of the free loop space
$LX=\mathbb{C}^{\infty}(S^1, X)$ with the circle $\mathbb{T}$
acting on $LX$ by rotating loops.

It is surprisingly difficult to make this precise, especially if
one wishes to consider equivariant generalization of this
construction. In this case the loop space $LX$  with the natural
rotation action is a rich orbifold. In this paper we offer a new
formulation between the loop space and Tate K-theory via a new
theory which we call quasi-elliptic cohomology.

Tate K-theory is the generalized elliptic cohomology associated to
the Tate curve. The Tate curve $Tate(q)$ is an  elliptic curve
over Spec$\mathbb{Z}((q))$, which is classified as the completion
of the algebraic stack of some nice generalized elliptic curves at
infinity. A good reference for $Tate(q)$ is Section 2.6 of
\cite{AHS}. We give a sketch of it in Section \ref{tatecurve}. The
relation between Tate K-theory and string theory is better
understood than for most known elliptic cohomology theories. The
definition of $G-$equivariant Tate K-theory for finite groups $G$
is modelled on the loop space of a global quotient orbifold, which
is formulated explicitly in Section 2, \cite{Gan07}. Its relation with string theory and loop space makes Tate K-theory itself a distinctive subject to study.

The idea of quasi-elliptic cohomology is  motivated by Ganter's
construction of Tate K-theory. 
It is not an elliptic cohomology  but from it we can recover the
Tate K-theory. This new theory can be interpreted in a neat form
by equivariant K-theories, which makes many constructions on it
easier and more natural than those on the Tate K-theories. Some
formulations can be generalized to other equivariant cohomology
theories. In addition, quasi-elliptic cohomology provides a method
that reduces facts such as the classification of geometric structures on the Tate curve into questions in representation theory.


\subsection{Loop Space}

Quasi-elliptic cohomology is  modelled on a version of equivariant
loop space. For background on orbifolds and Lie groupoids, we
refer the readers to Section 2, 3, \cite{LerStack} and
\cite{Moe02}.

For any  compact Lie group $G$ and a manifold $X$ with a smooth
$G-$action, there is a Lie groupoid $X/\!\!/G$ which is explained
in detail in Chapter 11, \cite{BTG}.
Smooth unbased loops in the orbifold $X/\!\!/G$ carries a lot of
structure: on the one hand, it includes loops represented by
smooth maps $\gamma: \mathbb{R}\longrightarrow X$ such that
$\gamma(t+1)=\gamma(t)g$ for some $g\in G$; other than the group
action by the loop group $LG:=\mathbb{C}^{\infty}(S^1, G)$, the
loop space also has the circle action by rotation. Lerman
discussed thoroughly in Section 3, \cite{LerStack} that the strict
2-category of Lie groupoids can be embedded into a weak 2-category
whose objects are Lie groupoids, 1-morphisms are bibundles and
2-morphisms equivariant diffeomorphisms between bibundles. Thus,
the free loop space of an orbifold $M$ is 
the category of bibundles from the trivial groupoid
$S^1/\!\!/\ast$ to the Lie groupoid $M$. We will write
$$Loop_1(X/\!\!/G):=Bibun(S^1/\!\!/\ast, X/\!\!/G),$$ which
is discussed in Definition \ref{loopspacemorphism}.
In Definition \ref{loopext3space}, we extend $Loop_1(X/\!\!/G)$ to
a groupoid $Loop_1^{ext}(X/\!\!/G)$ by adding rotations as
morphisms.

Especially we are interested in the ghost loops groupoid
$GhLoop(X/\!\!/G)$, which is defined to be the full subgroupoid of
$Loop^{ext}_1(X/\!\!/G)$ consisting of objects $(\pi, f)$ with the
image of $f$ contained in a single $G-$orbit. Ghost loops are
introduced by Rezk in his unpublished manuscript \cite{Rez16}. Another reference is Section 2.1.3, \cite{Huanthesis}.
This groupoid has several good properties. They are computed
locally in $X$. For instance, if $X=U\cup V$ where $U$ and $V$ are
$G-$invariant open subsets, then
$$GhLoop(X/\!\!/G)\cong GhLoop(U/\!\!/G)\cup_{GhLoop((U\cap
V)/\!\!/G)} GhLoop(V/\!\!/G).$$ So it satisfies a kind of
Mayer-Vietoris property. In addition, if $H$ is a closed subgroup
of $G$ and $X$ is the quotient space $G/H$, $GhLoop(X/\!\!/G)$ is
equivalent to $GhLoop(\mbox{pt}/\!\!/H)$. In other words, it has
the change-of-group property.

When $G$ is finite, $GhLoop(X/\!\!/G)$ is isomorphic to the full
subgroupoid $\Lambda(X/\!\!/G)$ of $Loop^{ext}_1(X/\!\!/G)$
consisting of constant loops. This groupoid $\Lambda(X/\!\!/G)$
can be regarded as an extended version of the inertia groupoid
$I(X/\!\!/G)$. Please see Definition \ref{inertiagroupoid} for
inertia groupoid.



\subsection{Quasi-elliptic cohomology}

For any compact orbifold groupoid $\mathbb{G}$, the orbifold
K-theory $K_{orb}(\mathbb{G})$ is defined to be the Grothendieck
ring of isomorphism classes of $\mathbb{G}-$vector bundles on
$\mathbb{G}$. In particular, $K_{orb}(X/\!\!/G)$ is $K_G(X)$. A
reference for orbifold K-theory is Chapter 3, \cite{ALRuan} and a
reference for equivariant K-theory is \cite{SegalequiK}.

Quasi-elliptic cohomology $QEll^*(X/\!\!/G)$ is defined to be the
orbifold K-theory of a subgroupoid $\Lambda(X/\!\!/G)$ of
$GhLoop(X/\!\!/G)$ consisting of constant loops.
When $G$ is a finite group,  $QEll^*_{G}(X)$ can be expressed in
terms of the equivariant K-theory of $X$ and its subspaces as
\begin{equation}QEll^*_G(X):=K_{orb}(GhLoop(X/\!\!/G))\cong \prod_{\sigma\in
G_{conj}}K^*_{\Lambda_G(\sigma)}(X^{\sigma})=\bigg(\prod_{\sigma\in
G}K^*_{\Lambda_G(\sigma)}(X^{\sigma})\bigg)^G,\label{defintro}\end{equation}
where $G_{conj}$ is a set of representatives of $G-$conjugacy
classes in $G$. The group $\Lambda_G(\sigma):=
C_G(\sigma)\times\mathbb{R}/\langle(\sigma, -1)\rangle$ acts on
the fixed point space $X^{\sigma}$ by $[g, t]\cdot m=g\cdot m$.
 In a coming paper by the author \cite{Huancoming}, we will present the construction of
$QEll^*_{G}(X)$ for any compact Lie group $G$.

$QEll_G(X)$ has the structure of a $\mathbb{Z}[q^{\pm}]-$algebra.
We have
\begin{equation}QEll^*_G(X)\otimes_{\mathbb{Z}[q^{\pm}]}\mathbb{Z}((q))=(K^*_{Tate})_G(X).
\label{tateqellequiv}\end{equation}

We formulate the K$\ddot{u}$nneth map, restriction map, change of
group isomorphism and transfer for $QEll$. In general, if $H^*$ is
an equivariant cohomology theory, then the functor $$X/\!\!/G
\mapsto H^*(GhLoop(X/\!\!/G))$$ gives a new equivariant cohomology
theory. Moreover, for each global cohomology theory, we can
formulate a new global cohomology theory via the ghost loops.

\subsection{Power operation}

One significant feature of quasi-elliptic cohomology is that it
has power operations, which was first observed by
Ganter, as shown in \cite{Gan07} and \cite{Gan13}. 
In Section \ref{poweroperation} we construct the total power
operation of quasi-elliptic cohomology. It satisfies the axioms
for equivariant power operations that Ganter gave in Definition
4.3 in \cite{Gan06}. For more details, please see Theorem
\ref{main1p}.

The power operation $\{\mathbb{P}_n\}_{n\geq 0}$ mixes the power
operation in $K-$theory with the natural operations of dilating
and rotating loops. The key point of the construction of the power
operation is an intermediate groupoid $d_{(\underline{g},
\sigma)}(X)$ with $(\underline{g}, \sigma)\in G\wr\Sigma_n$. It is
constructed from $\Lambda(X/\!\!/G)$ and isomorphic to $(X^{\times
n})^{(\underline{g}, \sigma)}/\!\!/
\Lambda_{G\wr\Sigma_n}(\underline{g}, \sigma)$.  
For more details of the construction, please see Section
\ref{s2complete}.

We illustrate what this power operation looks like by examples.
Let $G$ be the trivial group and $X$ a space. Let $(-)_k$ denote
the rescaling map defined in (\ref{lpok}).

When $n=2$, $\mathbb{P}_{(\underline{1}, (1)(1))}(x)=x\boxtimes x$
and $\mathbb{P}_{(\underline{1}, (12))}(x)=(x)_2$.

When $n=3$, $\mathbb{P}_{(\underline{1}, (1)(1)(1))}(x)=x\boxtimes
x\boxtimes x$, $\mathbb{P}_{(\underline{1},
(12)(1))}(x)=(x)_2\boxtimes x$, and $\mathbb{P}_{(\underline{1},
(123))}(x)=(x)_3$.

In these cases, the number of factors corresponds to the number of
cycles in the permutation and the rescaling map corresponds to the
length of each cycle. For more examples please see Example
\ref{pointsymmpower}.

For any equivariant cohomology theory $\{H^*_G( -)\}_G$ with an
$H_{\infty}$-structure in Ganter's sense, we can formulate a power
operation for the equivariant cohomology theories
$$\mathbb{H}^*_G(-):=\prod_{\sigma \in G_{conj}}
H^*_{\Lambda_G(\sigma)}(-)^{\sigma}$$ in the same way.

In addition, we can formulate the total power operation for the
orbifold quasi-elliptic cohomology in the sense of Definition 3.9,
\cite{Gan13}. The construction of the power operation is shown in
Section \ref{s2}.

\subsection{Classification of the finite subgroups of the Tate curve}

Though the general formulas for the power operations in $QEll_G$
are complicated, to understand it, it is useful to consider
special cases. It is already illuminating to consider the case
that $X$ is a point and $G$ is the trivial group, the power
operation has a neat form, as shown in Example
\ref{pointsymmpower}. It has a natural interpretation in terms of
the Tate elliptic curve.

In Section \ref{proofstrict} applying the power operation we prove
that the Tate K-theory of symmetric groups modulo the transfer
ideal classifies the finite subgroups of the Tate curve, which is
analogous to the principal result in Strickland \cite{Str98} that
the Morava $E-$theory of the symmetric group $\Sigma_n$ modulo a
certain transfer ideal classifies the power subgroups of rank $n$
of the formal group $\mathbb{G}_E$.

The finite subgroups of the Tate curve are classified by
$$\prod_{d|N}\mathbb{Z}((q))[q']/\langle q^d-q'^{\frac{N}{d}}
\rangle.$$

First we prove the parallel conclusion for quasi-elliptic
cohomology. \begin{theorem} \label{strickqell}
\begin{equation}QEll^0_{\Sigma_N}(\mbox{pt})/\mathcal{I}^{QEll}_{tr}\cong
\prod_{d|N}\mathbb{Z}[q^{\pm}][q']/\langle q^d-q'^{\frac{N}{d}}
\rangle,\label{ellc}\end{equation} where $\mathcal{I}^{QEll}_{tr}$
is the transfer ideal defined in (\ref{transferidealqec}) and $q'$
is the image of $q$ under the power operation $\mathbb{P}_N$. 
\end{theorem}

Then applying the relationship between $QEll^*$ and Tate K-theory,
we obtain the main theorem.
\begin{theorem}
The Tate K-theory of symmetric groups modulo the transfer ideal
$I^{Tate}_{tr}$  defined in (\ref{transferidealtatek}) classifies
finite subgroups of the Tate curve. Explicitly,
\begin{equation}(K^0_{Tate})_{\Sigma_N}(\mbox{pt})/I^{Tate}_{tr}\cong
\prod_{d|N}\mathbb{Z}((q))[q']/\langle q^d-q'^{\frac{N}{d}}
\rangle,\label{tatec}\end{equation} where $q'$ is the image of $q$
under the power operation $P^{Tate}$ constructed
in Definition 5.10, \cite{Gan07}. 

\end{theorem}

Moreover, via the isomorphism in Theorem \ref{strickqell}, we can
define a ring homomorphism
\begin{align*}\overline{P}_N: &QEll_{G}(X) \buildrel{\mathbb{P}_N}\over\longrightarrow QEll_{G\wr\Sigma_N}(X^{\times N})
\buildrel{res}\over\longrightarrow QEll_{G\times
\Sigma_N}(X^{\times N})\\  &\buildrel{diag^*}\over\longrightarrow
 QEll_{G\times\Sigma_N}(X) \cong QEll_G(X)\otimes_{\mathbb{Z}[q^{\pm}]} QEll_{\Sigma_N}(\mbox{pt}) \\
  &\longrightarrow QEll_G(X)\otimes_{\mathbb{Z}[q^{\pm}]} QEll_{\Sigma_N}(\mbox{pt})/\mathcal{I}^{QEll}_{tr},\end{align*}
 as shown in Proposition
\ref{adamsqell}. 
Under the identification (\ref{tateqellequiv}), it extends
uniquely to the ring homomorphism
$$\overline{P^{string}}_N: (K_{Tate})_G(X)\longrightarrow
(K_{Tate})_G(X)\otimes_{\mathbb{Z}((q ))}
(K_{Tate})_{\Sigma_N}(\mbox{pt})/I^{Tate}_{tr}$$ constructed in
Section 5.4, \cite{Gan07}. In \cite{HUgpT} we construct the
universal finite subgroup of the Tate curve via the operation
$\overline{P}_N$.

\subsection{Acknowledgement}
I would like to thank Charles Rezk who is always a wonderful
advisor and very inspiring teacher. Most of this work was guided
by him and it is a great experience to work with him. I would also
like to thank Matthew Ando. We had many mathematical discussions,
which are also important for my work.
At last, I would like to thank the editors and the referee for spending time on reading this work and
for their constructive and deep suggestion.

\section{Models for orbifold loops and ghost loops} \label{introduction}
To understand $QEll^*_G(X)$, it is essential to understand the
orbifold loop space. In this section, we will describe several
models for the loop space of $X/\!\!/G$. In Definition
\ref{loopspacemorphism} we discuss $Loop_1(X/\!\!/G)$ and
introduce another model $Loop_2(X/\!\!/G)$  in Definition
\ref{loopspace3}.

The groupoid structure of $Loop_1(X/\!\!/G)$ generalizes $Map(S^1,
X)/\!\!/G$, which is a subgroupoid of it. Other than the
$G-$action, we also
consider the rotation by the circle group $\mathbb{T}$ on the objects 
and form the  groupoids $Loop_1^{ext}(X/\!\!/G)$ and
$Loop_2^{ext}(X/\!\!/G)$. The groupoid $Loop_2^{ext}(X/\!\!/G)$
has a skeleton
$$\mathcal{L}(X/\!\!/G):=\coprod_{g}{_1}\mathcal{L}{_g}
X/\!\!/L^1_gG\rtimes\mathbb{T},$$ where each ${_1}\mathcal{L}{_g}
X =\mbox{Map}_{\mathbb{Z}/l\mathbb{Z}}(\mathbb{R}/l\mathbb{Z}, X)$
with $l$ the order of $g$ is equipped with an evident
$C_G(g)-$action. 
$\mathcal{L} (X/\!\!/G)$ has the same space of objects as the
groupoid $L(X/\!\!/G)$ discussed in Definition 2.3, \cite{LUloop},
from which equivariant Tate K-theory is defined. 
It has richer morphisms. The circle group $\mathbb{T}$ acts on
$\mathbb{R}/l\mathbb{Z}$ by rotation, and so in principle on the
orbifold ${_1}\mathcal{L}_gX$.

The key groupoid $\Lambda(X/\!\!/G)$ in the construction of
quasi-elliptic cohomology is the full subgroupoid of $\mathcal{L}
(X/\!\!/G)$ consisting of the constant loops.  In order to unravel
the relevant notations in the construction of $QEll^*_G(X)$, we
study the orbifold loop space 
in Section \ref{orbifoldloop} and Section \ref{ghost}.



In Section \ref{orbifoldspre} we define $Loop_1(X/\!\!/G)$. In
Section \ref{orbifoldloop} we interpret the enlarged groupoid
$Loop^{ext}_1(X/\!\!/G)$ and introduce a skeleton
$\mathcal{L}(X/\!\!/G)$ of it. In Section \ref{ghost} we show the
construction of quasi-elliptic cohomology by ghost loops. In
Section \ref{lambdarepresentationlemma} we show the representation
ring of $\Lambda_G(g)$. In Section \ref{orbqec} we introduce the
construction of quasi-elliptic cohomology first in terms of
orbifold K-theory and then equivariant K-theory. We show the
properties of the theory in Section \ref{propertiesqec}.


\subsection{Loop space} \label{loopspace}


\subsubsection{Bibundles}\label{orbifoldspre}

A standard reference for groupoids and bibundles is Section 2 and
3, \cite{LerStack}. For each pair of Lie groupoids $\mathbb{H}$
and $\mathbb{G}$,  the bibundles from $\mathbb{H}$ to
$\mathbb{G}$ are defined in Definition 3.25, \cite{LerStack}. The
category $Bibun(\mathbb{H}, \mathbb{G})$ has
 bibundles from $\mathbb{H}$ to $\mathbb{G}$ as the objects and bundle maps as the morphisms.

\begin{example}[$Bibun(S^1/\!\!/\ast, \ast/\!\!/G)$] According to the definition, a bibundle from $S^1/\!\!/\ast$
to $\ast/\!\!/G$ with $G$ a Lie group is a smooth manifold $P$
together with two maps $\pi: P\longrightarrow S^1$ a smooth
principal $G-$bundle and the constant map $r: P\longrightarrow
\ast$.  So a bibundle in this case is equivalent to a smooth
principal $G-$bundle over $S^1$. The morphisms in
$Bibun(S^1/\!\!/\ast, \ast/\!\!/G)$ are bundle isomorphisms.
\label{babyloop1}
 \end{example}

\begin{definition}[$Loop_1(X/\!\!/G)$]

Let $G$ be a Lie group acting smoothly on a manifold $X$. We use
$Loop_1(X/\!\!/G)$ to denote the category $Bibun(S^1/\!\!/\ast,
X/\!\!/G)$, which generalizes Example \ref{babyloop1}. Each object
consists of a smooth manifold $P$ and two structure maps
$P\buildrel{\pi}\over\longrightarrow S^1$ a smooth principal
$G-$bundle  and $f: P\longrightarrow X$ a $G-$equivariant map. We
use the same symbol $P$ to denote both the object and the smooth
manifold when there is no confusion. A morphism 
 is a $G-$bundle map $\alpha: P\longrightarrow P'$ making the
diagram below commute.
$$\xymatrix{S^1
&P \ar[l]_{\pi}\ar[d]^{\alpha}\ar[r]^{f} & X \\
&P' \ar[lu]^{\pi'}\ar[ru]_{f'} &}$$ Thus, the morphisms in
$Loop_1(X/\!\!/G)$ from $P$ to $P'$ are bundle isomorphisms.



\label{loopspacemorphism}\end{definition}

Only the $G-$action on $X$ is considered in $Loop_1(X/\!\!/G)$. We
add the rotations by adding more morphisms into the groupoid.

\begin{definition}[$Loop^{ext}_1(X/\!\!/G)$]\label{loopext3space} 
Let $Loop^{ext}_1(X/\!\!/G)$ denote the groupoid with the same
objects as $Loop_1(X/\!\!/G)$. Each morphism
consists of the pair $(t, \alpha)$ where $t\in\mathbb{T}$ is a
rotation and $\alpha$ is a $G-$bundle map. They make the diagram
below commute.
$$\xymatrix{S^1\ar[d]_{t}
&P \ar[l]_{\pi}\ar[d]_{\alpha}\ar[r]^{f} & X \\S^1 &P'
\ar[l]^{\pi'}\ar[ru]_{f'} &}$$

The groupoid $Loop_1(X/\!\!/G)$ is a subgroupoid of
$Loop^{ext}_1(X/\!\!/G)$.
\end{definition}

\subsubsection{Another model for orbifold loop space}\label{orbifoldloop}

We give an equivalent description of the groupoids discussed in
Section \ref{orbifoldspre}. The new models $Loop_2(X/\!\!/G)$ and
$Loop_2^{ext}(X/\!\!/G)$ are more practicable to compute. We give
a skeleton $\mathcal{L}(X/\!\!/G)$ of $Loop_2^{ext}(X/\!\!/G)$
when $G$ is finite in Proposition \ref{loop1equivske}.


\begin{definition}[$Loop_2(X/\!\!/G)$] \label{loopspace3}
Let $Loop_2(X/\!\!/G)$ denote the groupoid whose objects are
$(\sigma, \gamma)$ with $\sigma\in G$ and $\gamma:
\mathbb{R}\longrightarrow X$
a continuous map such that $\gamma(s+1)= \gamma(s)\cdot\sigma$, for any $s\in\mathbb{R}$. 
A morphism $\alpha: (\sigma, \gamma)\longrightarrow (\sigma',
\gamma')$ is a continuous map $\alpha: \mathbb{R}\longrightarrow
G$ satisfying $\gamma'(s)= \gamma(s)\alpha(s)$. Note that
$\alpha(s)\sigma'=\sigma\alpha(s+1)$, for any $s\in\mathbb{R}$.
\end{definition}

Moreover, we can extend the groupoid $Loop_2(X/\!\!/G)$ by adding
the rotations.

\begin{definition}[$Loop^{ext}_2(X/\!\!/G)$]\label{loopext3space2}

Let $Loop^{ext}_2(X/\!\!/G)$ denote the groupoid with the same
objects as $Loop_2(X/\!\!/G)$. A morphism $(\sigma,
\gamma)\longrightarrow (\sigma', \gamma')$ consists of the pair
$(\alpha, t)$ with $\alpha:\mathbb{R}\longrightarrow G$ a
continuous map and $t\in\mathbb{R}$ 
satisfying $\gamma'(s)=\gamma(s-t)\alpha(s-t)$. Note that
$(\alpha, t+1)$ and $(\alpha\sigma', t)$ are the same morphism and
each morphism can be represented by a pair $(\alpha, t)$ with
$t\in[0, 1)$.

$Loop_2(X/\!\!/G)$ is a subgroupoid of $Loop^{ext}_2(X/\!\!/G)$.

\end{definition}

\begin{lemma}The groupoid $Loop^{ext}_1(X/\!\!/G)$ is isomorphic to a full subgroupoid of $Loop^{ext}_2(X/\!\!/G)$.\end{lemma}

\begin{proof}
Define a functor  $$F: Loop^{ext}_1(X/\!\!/G)\longrightarrow
Loop^{ext}_2(X/\!\!/G)$$ by sending an object
$$\begin{CD} S^1 @<{\pi}<< P @>{f}>> X\end{CD}$$ to $(\sigma,
\gamma)$ with $\gamma(s):= f([s, e])$ and
$\sigma=\gamma(0)^{-1}\gamma(1)$ and sending a morphism
$$\xymatrix{S^1\ar[d]_{t} &P \ar[l]_{\pi}\ar[d]^F\ar[r]^{f} & X
\\ S^1 &P' \ar[l]^{\pi'}\ar[ru]_{f'} &}$$ to $(\alpha, t): (\sigma,
\gamma)\longrightarrow (\sigma', \gamma')$ with $\alpha(s):=F([s,
e])^{-1}.$

$F$ is a fully faithful functor but not essentially surjective.

\end{proof}

Therefore, the groupoid $Loop^{ext}_2(X/\!\!/G)$ contains all the
information of $Loop^{ext}_1(X/\!\!/G)$. Next we will show a
skeleton of this larger groupoid when $G$ is finite. Before that,
we introduce some symbols.

Let $k\geq 0$ be an integer and $g$ an element in the compact Lie
group $G$. Let $L^k_{g}G$ denote the twisted loop group
\begin{equation}
\{\gamma: \mathbb{R}\longrightarrow
G|\gamma(s+k)=g^{-1}\gamma(s)g\}.
\end{equation}
The multiplication of it is defined by
\begin{equation}(\delta\cdot\delta')(t)=\delta(t)\delta'(t)\mbox{,     for      any     }\delta, \delta'\in
L^k_{g}G,  \mbox{    and     }t\in\mathbb{R}.\end{equation} The
identity element $e$ is the constant map sending all the real
numbers to the identity element of $G$. We extend this group by
adding the rotations. Let $L^k_{g}G\rtimes\mathbb{T}$ denote the
group with elements $(\gamma, t)$, $\gamma\in L^k_{g}G$ and
$t\in\mathbb{T}$. The multiplication is defined by
\begin{equation} (\gamma, t)\cdot (\gamma', t'):= (s\mapsto \gamma(s) \gamma'(s+t), t+t').\label{lkgtmulti}\end{equation}

The set of constant maps $\mathbb{R}\longrightarrow G$ in
$L^k_{g}G$ is a subgroup of it, i.e. the centralizer $C_G(g)$.
When $G$ is finite, $L^k_{g}G=C_G(g)$.

\bigskip

When $G$ is finite, the objects of $Loop_2(X/\!\!/G)$ can be
identified with the space
$$\coprod\limits_{g\in G}{_1}\mathcal{L}{_g} X$$ where
\begin{equation}{_k}\mathcal{L}{_g}
X:=\mbox{Map}_{\mathbb{Z}/l\mathbb{Z}}(\mathbb{R}/kl\mathbb{Z},
X),\label{chongL}\end{equation} and $l$ is the order of the
element $g$. The cyclic group $\mathbb{Z}/l\mathbb{Z}$ is
isomorphic to the subgroup $k\mathbb{Z}/kl\mathbb{Z}$ of
$\mathbb{R}/kl\mathbb{Z}$. The isomorphism
$\mathbb{Z}/l\mathbb{Z}\longrightarrow k\mathbb{Z}/kl\mathbb{Z}$
sends the generator $[1]$ corresponding to $1$ to the generator
$[k]$ of $k\mathbb{Z}/kl\mathbb{Z}$ corresponding to $k$.
$k\mathbb{Z}/kl\mathbb{Z}$ acts on $\mathbb{R}/kl\mathbb{Z}$ by
group multiplication. Thus, via the isomorphism,
$\mathbb{Z}/l\mathbb{Z}$ acts on $\mathbb{R}/kl\mathbb{Z}$.
$\mathbb{Z}/l\mathbb{Z}$ is also isomorphic to the cyclic group
$\langle g\rangle$ by identifying the generater $[1]$ with $g$. So
it acts on $X$ via the $G-$action on it.

${_1}\mathcal{L}{_g} X/\!\!/ L^1_gG$ is a full subgroupoid of
$Loop_2(X/\!\!/G)$. Moreover, ${_1}\mathcal{L}{_g} X/\!\!/
L^1_gG\rtimes\mathbb{T}$ is
a full subgroupoid of $Loop^{ext}_2(X/\!\!/G)$ where  
$L^k_{g}G\rtimes\mathbb{T}$ acts on ${_k}\mathcal{L}{_g} X$ by
\begin{equation}\delta \cdot(\gamma, t):= (s\mapsto
\delta(s+t)\cdot\gamma(s+t)), \mbox{  for any }(\gamma, t)\in
L^k_gG\rtimes\mathbb{T},\mbox { and    }\delta \in
{_k}\mathcal{L}{_g} X.\label{chongaction}\end{equation} The action
by $g$ on ${_k}\mathcal{L}{_g} X$ coincides with that by
$k\in\mathbb{R}$. So we have the isomorphism
\begin{equation}
L^k_gG\rtimes\mathbb{T}=L_g^k G\rtimes \mathbb{R}/\langle
(\overline{g}, -k)\rangle,\label{bengkui}
\end{equation} where $\overline{g}$ represents the constant loop
$\mathbb{T}\longrightarrow \{g\}\subseteq G$.

In fact we have already proved Proposition \ref{loop1equivske}.
\begin{proposition}
Let $G$ be a finite group. The groupoid
$$\mathcal{L}(X/\!\!/G):=\coprod_{[g]}{_1}\mathcal{L}{_g}
X/\!\!/L^1_gG\rtimes\mathbb{T}$$ is a skeleton of
$Loop^{ext}_2(X/\!\!/G)$, where the coproduct goes over conjugacy
classes in $\pi_0G$.   \label{loop1equivske}
\end{proposition}

Next we show the physical meaning of  $L_{\sigma}^1 G$. Recall
that the gauge group of a principal bundle is defined to be the
group of its vertical automorphisms. The readers may refer
\cite{MV} for more details. For a $G-$bundle $P\longrightarrow
S^1$, let $L_P G$ denote its gauge group.

We have the well-known facts below. 
\begin{lemma}The principal $G-$bundles over $S^1$ are classified up to isomorphism by
homotopy classes $$[S^1, BG]\cong \pi_0G/\mbox{conj}.$$ Up to
isomorphism every principal $G-$bundle over $S^1$ is isomorphic to
one of the forms $P_{\sigma}\longrightarrow S^1$ with $\sigma\in
G$ and
$$P_{\sigma}:=\mathbb{R}\times G/ (s+1, g)\sim(s, \sigma g).$$ A complete collection of isomorphism classes is given by  a
choice of representatives for each conjugacy class of
$\pi_0G$.\label{ghostlem1}\end{lemma}

For the gauge group $L_{P_{\sigma}} G$ we have the conclusion
below.

\begin{proposition}For the bundle
$P_{\sigma}\longrightarrow S^1$, $L_{P_{\sigma}}G$ is isomorphic
to the twisted loop group $L^1_{\sigma}G$.
\label{ghostlem2}\end{proposition}

\begin{proof}


Each automorphism $f$ of the bundle $P_{\sigma}\longrightarrow
S^1$ has the form
\begin{equation}\begin{CD}P_{\sigma} @>{[s, g]\mapsto [s, \gamma_f(s)g]}>> P_{\sigma} \\ @VVV @VVV \\ S^1 @>{=}>> S^1
\end{CD}\end{equation} for some $\gamma_f: \mathbb{R}\longrightarrow
G$.  The morphism is well-defined if and only if
$\gamma_f(s+1)=\sigma^{-1}\gamma_f(s)\sigma$. So we get a
well-defined map $$F: L_{P_{\sigma}}G\longrightarrow
L^1_{\sigma}G\mbox{,    } f\mapsto\gamma_f.$$  It is a bijection.
Moreover, by the property of group action, $F$ sends the identity
map to the constant map $\mathbb{R}\longrightarrow G\mbox{,   }
s\mapsto e$, which is the trivial element in $L^1_{\sigma}G$, and
for two automorphisms $f_1$ and $f_2$ at the object, $F(f_1\circ
f_2)= \gamma_{f_1}\cdot \gamma_{f_2}$. So $L_{P_{\sigma}}G$ is
isomorphic to $ L^1_{\sigma}G$.

\end{proof}

\subsubsection{Ghost Loops} \label{ghost}

Let $G$ be a compact Lie group and $X$ a $G-$space. In this
section we introduce a subgroupoid $GhLoop(X/\!\!/G)$ of
$Loop^{ext}_1(X/\!\!/G)$, which can be computed locally.

\begin{definition}[Ghost Loops]The groupoid of ghost loops is defined to be the full subgroupoid
$GhLoop(X/\!\!/G)$ of $Loop^{ext}_1(X/\!\!/G)$ consisting of
objects $S^1\leftarrow
P\buildrel{\widetilde{\delta}}\over\rightarrow X$ such that
$\widetilde{\delta}(P)\subseteq X$ is contained in a single
$G-$orbit. \label{ghostloopdef} \end{definition}

For a given $\sigma\in G$, define the space
\begin{equation}
GhLoop_{\sigma}(X/\!\!/G):=\{\delta\in {_1}\mathcal{L}{_{\sigma}}
X |\delta(\mathbb{R})\subseteq G\delta(0) \}.
\end{equation}

We have a corollary of Proposition \ref{loop1equivske} below.
\begin{proposition}$GhLoop(X/\!\!/G)$ is equivalent to the
groupoid
$$\Lambda(X/\!\!/G):=\coprod_{[\sigma]}GhLoop_{\sigma}(X/\!\!/G)/\!\!/L^1_{\sigma}G\rtimes\mathbb{T}$$
where the coproduct goes over conjugacy classes in $\pi_0G$.
\label{skullofghost}
\end{proposition}

\begin{example}
If $G$ is a finite group, it has the discrete topology. In this
case, $LG$ consists of constant loops and, thus,  is isomorphic to
$ G$. The space of objects of $GhLoop(X/\!\!/G)$ can be identified
with $X$. For $\sigma\in G$ and any integer $k$, $L^k_{\sigma} G$
can be identified with $C_G(\sigma)$;
$L^k_{\sigma}G\rtimes\mathbb{T}\cong C_G(\sigma)\times
\mathbb{R}/\langle (\sigma, -k)\rangle$; and
$GhLoop_{\sigma}(X/\!\!/G)$ can be identified with $X^{\sigma}$.

\label{finiteghost}
\end{example}

Unlike true loops, ghost loops have the property that they can be
computed locally, as shown in  the lemma below. The proof is left
to the readers.

\begin{proposition}If $X=U\cup V$ where $U$ and $V$ are $G-$invariant open subsets,
then $GhLoop(X/\!\!/G)$ is isomorphic to the fibred product of
groupoids $$GhLoop(U/\!\!/G)\cup_{GhLoop((U\cap
V)/\!\!/G)} GhLoop(V/\!\!/G).$$ \label{ghostmv} 
\end{proposition}
Thus, the ghost loop construction satisfies Mayer-Vietoris
property. Moreover, it has the change-of-group property.

\begin{proposition}Let $H$ be a closed subgroup of $G$. It acts on the
space of left cosets $G/H$ by left multiplication. Let $\mbox{pt}$
denote the single point space with the trivial $H-$action. Then we
have the equivalence  of topological groupoids between
$Loop^{ext}_1((G/H)/\!\!/G)$ and $Loop^{ext}_1(\mbox{pt}/\!\!/H)$.
Especially, there is an equivalence between the groupoids
$GhLoop((G/H)/\!\!/G)$ and $GhLoop(\mbox{pt}/\!\!/ H)$.
\label{globalghost}\end{proposition}

\begin{proof}

First we define a functor $F:
Loop^{ext}_1((G/H)/\!\!/G)\longrightarrow
Loop^{ext}_1(\mbox{pt}/\!\!/ H)$ sending an object $S^1\leftarrow
P\buildrel{\widetilde{\delta}}\over\rightarrow G/H$  to
$S^1\leftarrow Q\rightarrow \{eH\}=\mbox{pt}$ where
$Q\longrightarrow eH$ is the constant map, and $Q\longrightarrow
S^1$ is the pull back bundle $$\xymatrix{Q \ar[r] \ar[d] &\{eH\}
\ar@{^{(}->}[d]
\\ P\ar[r] &G/H.}$$

It sends a morphism $$\xymatrix{P'\ar[r] \ar[d]&P \ar[r] \ar[d]
&G/H\\ S^1\ar[r] &S^1&}$$ to the morphism
$$\xymatrix{Q' \ar[r] \ar[d] &Q\ar[r] \ar[d] &\{eH\} \ar[d] \\ P'\ar[r]\ar[d] &P\ar[r]\ar[d] &G/H\\
S^1 \ar[r] &S^1 & } $$ where all the squares are pull-back.

In addition, we can define a functor $F': Loop^{ext}_1
(\mbox{pt}/\!\!/ H) \longrightarrow Loop^{ext}_1((G/H)/\!\!/G)$
sending an object $S^1\leftarrow
 Q\rightarrow \mbox{pt}$ to $S^1\leftarrow G\times_HQ\rightarrow
 G\times_H\mbox{pt}=G/H$ and sending a morphism $$\xymatrix{Q'\ar[r] \ar[d] & Q \ar[d] \\ S^1\ar[r]
 &S^1}$$ to $$\xymatrix{G\times_HQ'\ar[r] \ar[d] & G\times_HQ \ar[d] \ar[r] &G\times_H\mbox{pt}=G/H\\ S^1\ar[r]
 &S^1 &}$$

$F\circ F'$ and $F'\circ F$ are both identity maps. So the
topological groupoids $Loop^{ext}_1((G/H)/\!\!/G)$ and
$Loop^{ext}_1(\mbox{pt}/\!\!/ H)$ are equivalent.

\bigskip

We can prove the equivalence between $GhLoop((G/H)/\!\!/G)$ and
$GhLoop(\mbox{pt}/\!\!/ H)$ in the same way.
\end{proof}


\begin{remark}In general, if $H^*$ is an equivariant cohomology theory,
Proposition \ref{globalghost} implies the functor $$X/\!\!/G
\mapsto H^*(GhLoop(X/\!\!/G))$$ gives a new equivariant cohomology
theory. When $H^*$ has the change of group isomorphism, so does
$H^*(GhLoop(-))$.
\end{remark}

\section{Quasi-elliptic cohomology $QEll^*_G$}\label{conqec}

Unless otherwise indicated, we assume $G$ is a finite group and
$X$ is a $G-$space in the rest part of the paper. The main
references for Section \ref{conqec} are Rezk's unpublished work \cite{Rez11} and the author's PhD thesis \cite{Huanthesis}. The construction of the theory $QEll^*_G$
for any compact Lie group $G$ will be shown in the paper \cite{Huancoming}. In Section
\ref{orbqec} we define $QEll^*_G$ and prove some of its main
properties.
Before that we discuss in Section \ref{lambdarepresentationlemma}
the complex representation ring of
\begin{equation}\Lambda_G(g):=L^1_{g}G\rtimes\mathbb{T}\cong C_G(g)\times
\mathbb{R}/\langle (g, -1)\rangle,\label{lambdadef}\end{equation}
which is a factor of $QEll^*_G(\mbox{pt})$. We assume familiarity
with \cite{SegalequiK} and \cite{KC}.

\subsection{Preliminary: representation ring of
$\Lambda_G(g)$}\label{lambdarepresentationlemma}
 Let $q: \mathbb{T}\longrightarrow U(1)$
be the isomorphism $t\mapsto e^{2\pi it}$. The complex
representation ring $R\mathbb{T}$ is $\mathbb{Z}[q^{\pm}]$.

We have an exact sequence
$$1\longrightarrow C_G(g)\longrightarrow
\Lambda_G(g)\buildrel{\pi}\over\longrightarrow\mathbb{T}\longrightarrow
0$$ where the first map is $g\mapsto [g, 0]$ and the second map is
\begin{equation}\pi([g, t])= e^{2\pi it}.\label{pizq}\end{equation}
The map $\pi^*: R\mathbb{T}\longrightarrow R\Lambda_G(g)$ equips
the representation ring $R\Lambda_G(g)$ the structure as an $R\mathbb{T}-$module.

There is a relation between the complex representation ring of
$C_G(g)$ and that of $\Lambda_G(g)$, which is shown as Lemma 1.2
in \cite{Rez11} and Lemma 2.4.1 in \cite{Huanthesis}.

\begin{lemma}

The $R\mathbb{T}-$module $R\Lambda_G(g)$ with the action defined by $\pi^*: R\mathbb{T}\longrightarrow R\Lambda_G(g)$ is a free module.

In particular, there is an $R\mathbb{T}-$basis of $R\Lambda_G(g)$
given by irreducible representations $\{V_{\lambda}\}$, such that
restriction $V_{\lambda}\mapsto V_{\lambda}|_{C_G(g)}$ to $C_G(g)$
defines a bijection between $\{V_{\lambda}\}$ and the set
$\{\lambda\}$ of irreducible representations of
$C_G(g)$.\label{cl}
\end{lemma}
\begin{proof}

Let $l$ be the order of  $g$. Note that $\Lambda_G(g)$ is
isomorphic to
$$C_G(g)\times\mathbb{R}/l\mathbb{Z}/\langle(g, -1)\rangle.$$ Thus, it is the quotient of the product of two compact Lie groups.

Let $\lambda: C_G(g)\longrightarrow GL(n, \mathbb{C})$ be an
$n-$dimensional $C_G(g)-$representation with representation space
$V$ and $\eta: \mathbb{R}\longrightarrow GL(n, \mathbb{C})$ be a
representation of $\mathbb{R}$ such that $\lambda(g)$ acts on $V$
via scalar multiplication by $\eta(1)$. Define
 a $n-$dimensional
$\Lambda_G(g)-$representation $\lambda\odot_{\mathbb{C}} \eta$
with representation space $V$ by
\begin{equation}\lambda\odot_{\mathbb{C}} \eta([h, t]):
=\lambda(h)\eta(t).\end{equation}

Any irreducible $n-$dimensional representation of the quotient
group $\Lambda_G(g)=C_G(g)\times\mathbb{R}/\langle(g, -1)\rangle$
is an irreducible $n-$dimensional representation of the product
$C_G(g)\times\mathbb{R}$. And any finite dimensional irreducible
complex representation of the product of two compact Lie groups is
the tensor product of an irreducible representation of each
factor. So any irreducible representation of the quotient group
$\Lambda_G(g)$ is the tensor product of an irreducible
representation $\lambda$ of $C_G(g)$ with representation space $V$
and an irreducible representation $\eta$ of $\mathbb{R}$. Any
irreducible complex representation $\eta$ of $\mathbb{R}$ is one
dimensional. So the representation space of
$\lambda\odot_{\mathbb{C}} \eta$ is still $V$. $\eta(1)^l=I$. We
need $\eta(1)=\lambda (g)$. So $\eta(1)=e^{\frac{2\pi ik}{l}}$ for
some $k\in \mathbb{Z}$. So
$$\eta(t)= e^{\frac{2\pi i(k+lm)t}{l}}.$$ Any $m\in\mathbb{Z}$
gives a choice of  $\eta$ in this case. And $\eta$ is a
representation of $\mathbb{R}/l\mathbb{Z}\cong \mathbb{T}$.

Therefore, we have a bijective correspondence between

(1) isomorphism classes of irreducible
$\Lambda_G(g)-$representation $\rho$, and

(2) isomorphism classes of pairs $(\lambda, \eta)$ where $\lambda$
is an irreducible $C_G(g)-$representation and
$\eta:\mathbb{R}\longrightarrow \mathbb{C}^*$ is a character such
that $\lambda(g)=\eta(1)I$. $\lambda=\rho|_{C_G(g)}$.

Then as a corollary, the $R\mathbb{T}-$module $R\Lambda_G(g)$ with the $R\mathbb{T}-$action defined by $\pi^*: R\mathbb{T}\longrightarrow
R\Lambda_G(g)$

$\pi^*: R\mathbb{T}\longrightarrow
R\Lambda_G(g)$ exhibits $R\Lambda_G(g)$ as a free
$R\mathbb{T}-$module.
\end{proof}

\begin{remark} We can make a canonical choice of $\mathbb{Z}[q^{\pm}]$-basis
for $R\Lambda_{G}(g)$. For each irreducible $G$-representation
$\rho: G\longrightarrow Aut(G)$, write $\rho(\sigma)=e^{2\pi
ic}id$ for $c\in[0,1)$, and set $\chi_{\rho}(t)=e^{2\pi ict}$.
Then the pair $(\rho, \chi_{\rho})$ corresponds  to a unique
irreducible $\Lambda_{G}(g)$-representation
\begin{equation}\rho\odot_{\mathbb{C}} \chi_{\rho}([h, t]):
=\rho(h)\chi_{\rho}(t).\label{lambdaeq}\end{equation}
\label{lambdabasis}\end{remark}

\begin{example}[$G=\mathbb{Z}/N\mathbb{Z}$]
Let $G=\mathbb{Z}/N\mathbb{Z}$ for $N\geq 1$, and let $\sigma\in
G$. Given an integer $k\in\mathbb{Z}$ which projects to
$\sigma\in\mathbb{Z}/N\mathbb{Z}$, let $x_k$ denote the
representation of $\Lambda_G(\sigma)$ defined by
\begin{equation}\begin{CD}\Lambda_{G}(\sigma)=(\mathbb{Z}\times\mathbb{R})/(\mathbb{Z}(N,0)+\mathbb{Z}(k,1))
@>{[a,t]\mapsto[(kt-a)/N]}>> \mathbb{R}/\mathbb{Z}=\mathbb{T}
@>{q}>> U(1).\end{CD}\label{xk}\end{equation} $R\Lambda_G(\sigma)$
is isomorphic to the ring $\mathbb{Z}[q^{\pm}, x_k]/(x^N_k-q^k)$.
\label{ppex}
\end{example}

\begin{example}[$G=\Sigma_3$]\label{replambdasymm3}
$G=\Sigma_3$ has three conjugacy classes represented by $1$,
$(12)$, $(123)$ respectively.
\bigskip

$\Lambda_{\Sigma_3}(1)=\Sigma_3\times\mathbb{T}$, thus,
$R\Lambda_{\Sigma_3}(1)=R\Sigma_3\otimes R\mathbb{T}=\mathbb{Z}[X,
Y]/(XY-Y, X^2-1, Y^2-X-Y-1)\otimes\mathbb{Z}[q^{\pm}]$ where $X$
is the sign representation on $\Sigma_3$ and $Y$ is the standard
representation.
\bigskip

$C_{\Sigma_3}((12))=\langle(12)\rangle=\Sigma_2,$ thus,
$\Lambda_{\Sigma_3}((12))\cong\Lambda_{\Sigma_2}((12)).$ So we
have $$R\Lambda_{\Sigma_3}((12))\cong
R\Lambda_{\Sigma_2}((12))=\mathbb{Z}[q^{\pm},
x_1]/(x_1^2-q)\cong\mathbb{Z}[q^{\pm\frac{1}{2}}].$$

$C_{\Sigma_3}(123)=\langle(123)\rangle=\mathbb{Z}/3\mathbb{Z},$
thus,
$\Lambda_{\Sigma_3}((123))\cong\Lambda_{\mathbb{Z}/3\mathbb{Z}}(1).$
So we have $$R\Lambda_{\Sigma_3}((123))\cong\mathbb{Z}[q^{\pm},
x_1]/(x_1^3-q)\cong\mathbb{Z}[q^{\pm\frac{1}{3}}].$$\end{example}

\bigskip

Moreover, we have the conclusion below about the relation between
the induced representations
$Ind|^{\Lambda_G(\sigma)}_{\Lambda_H(\sigma)}(-)$ and
$Ind|^{C_G(\sigma)}_{C_H(\sigma)}(-).$
\begin{lemma}
Let $H$ be a subgroup of $G$ and $\sigma$ an element of $H$. Let
$m$ denote $[C_G(\sigma):C_H(\sigma)]$. Let $V$ denote a
$\Lambda_H(\sigma)-$representation $\lambda\odot_{\mathbb{C}}\chi$
with $\lambda$ a $C_H(\sigma)-$representation, $\chi$ a
$\mathbb{R}-$representation and $\odot_{\mathbb{C}}$ defined in
(\ref{lambdaeq}).

(i)
\begin{equation} res^{\Lambda_G(\sigma)}_{\Lambda_H(\sigma)}(\lambda\odot_{\mathbb{C}}\eta)=(res^{C_G(\sigma)}_{C_H(\sigma)}\lambda)\odot_{\mathbb{C}}\eta.\end{equation}

(ii) The induced representation
$$Ind^{\Lambda_G(\sigma)}_{\Lambda_H(\sigma)}
(\lambda\odot_{\mathbb{C}}\chi)$$ is isomorphic to the
$\Lambda_G(\sigma)-$representation
$$(Ind^{C_G(\sigma)}_{C_H(\sigma)}\lambda)\odot_{\mathbb{C}}\chi.$$
Their underlying vector spaces are both $V^{\oplus m}$.

Thus, the computation of both
$Ind^{\Lambda_G(\sigma)}_{\Lambda_H(\sigma)}
(\lambda\odot_{\mathbb{C}}\chi)$ and
$res^{\Lambda_G(\sigma)}_{\Lambda_H(\sigma)}(\lambda\odot_{\mathbb{C}}\eta)$
can be reduced to the computation of representations of finite
groups. \label{induequ}
\end{lemma}

The proof is straightforward and left to the readers.

\bigskip

Let $k$ be any integer. Next we describe the relation between
\begin{equation}\Lambda^k_G(g):=L^k_{g}G\rtimes\mathbb{T}\cong C_G(g)\times
\mathbb{R}/\langle (g, -k)\rangle \label{lambdakdef}\end{equation}
and $\Lambda_G(g)$, which gives the relation between their
representation rings.

There is an exact sequence
$$\begin{CD}1 @>>> C_G(g) @>{g\mapsto [g, 0]}>> \Lambda^k_G(g) @>{\pi_k}>>\mathbb{R}/k\mathbb{Z} @>>> 0\end{CD}$$
where the second map $\pi_k: \Lambda^k_G(g)\longrightarrow
\mathbb{R}/k\mathbb{Z}$ is $\pi_k([g, t])= e^{2\pi i t}$.

Let $q^{\frac{1}{k}}: \mathbb{R}/k\mathbb{Z}\longrightarrow U(1)$
denote the composition
$$\begin{CD}\mathbb{R}/k\mathbb{Z} @>{t\mapsto \frac{t}{k}}>> \mathbb{R}/\mathbb{Z} @>{q}>> U(1).\end{CD}$$
The representation ring $R(\mathbb{R}/k\mathbb{Z})$ is
$\mathbb{Z}[q^{\pm\frac{1}{k}}]$.

Analogous to Lemma \ref{cl}, we have the conclusion about
$R\Lambda^k_G(g)$ below.
\begin{lemma}\label{lambdakrepresentation}
The map $\pi^*_k: R(\mathbb{R}/k\mathbb{Z})\longrightarrow
R\Lambda^k_G(g)$ exhibits it as a free
$\mathbb{Z}[q^{\pm\frac{1}{k}}]-$module. There is a
$\mathbb{Z}[q^{\pm\frac{1}{k}}]-$basis of $R\Lambda^k_G(g)$ given
by irreducible representations $\{\rho_k\}$ such that the
restrictions $\rho_k|_{C_G(g)}$ of them to $C_G(g)$ are precisely
the $\mathbb{Z}$-basis of $RC_G(g)$ given by irreducible
representations.

In other words, any irreducible $\Lambda^k_G(g)-$representation
has the form $\rho\odot_{\mathbb{C}} \chi$ where $\rho$ is an
irreducible representation of $C_G(g)$,
$\chi:\mathbb{R}/k\mathbb{Z}\longrightarrow GL(n, \mathbb{C})$
such that $\chi(k)=\rho(g)$, and
\begin{equation}\rho\odot_{\mathbb{C}} \chi([h, t]):
=\rho(h)\chi(t)\mbox{,     for any   }[h, t]\in
\Lambda^k_G(g).\end{equation}
\end{lemma}

$R\Lambda_G^k(g)$ is a $\mathbb{Z}[q^{\pm}]-$module via the
inclusion $\mathbb{Z}[q^{\pm}]\longrightarrow
\mathbb{Z}[q^{\pm\frac{1}{k}}]$.

By Lemma \ref{lambdakrepresentation}, we can make a
$\mathbb{Z}[q^{\pm\frac{1}{k}}]$-basis
$\{\rho\odot_{\mathbb{C}}\chi_{\rho, k}\}$ for $R\Lambda^k_{G}(g)$
with each $\rho: G\longrightarrow Aut(G)$ an irreducible
$G$-representation and $\chi_{\rho, k}(t)= e^{2\pi i\frac{ct}{k}}$
with $c\in[0,1)$ such that $\rho(\sigma)=e^{2\pi ic}id$. This
collection $\{\rho\odot_{\mathbb{C}}\chi_{\rho, k}\}$ gives  a
$\mathbb{Z}[q^{\pm\frac{1}{k}}]-$basis of $R\Lambda^k_G(g)$.

\bigskip

There is a group isomorphism $\alpha_k:
\Lambda^k_G(g)\longrightarrow \Lambda_G(g)$ sending $[g, t]$ to
$[g, \frac{t}{k}]$. Observe that there is a pullback square of
groups
\begin{equation}\xymatrix{&\Lambda_G^k(g)\ar[r]^{\alpha_k}\ar[d]^{\pi_k}
&\Lambda_G(g)\ar[d]^{\pi}\\&\mathbb{R}/k\mathbb{Z}
\ar[r]^{t\mapsto
\frac{t}{k}}&\mathbb{R}/\mathbb{Z}}\label{alphaklambdagroup}\end{equation}

So we have the commutative square of a pushout square in the
category of $\lambda-$rings. \begin{equation}\xymatrix{&
R\Lambda^k_G(g)&R\Lambda_G(g)\ar[l]\\
&R(\mathbb{R}/k\mathbb{Z})\ar[u]&R\mathbb{T}\ar[u]\ar[l]}\label{alphaklambdaring}\end{equation}
It gives a canonical isomorphism of $\lambda-$rings
$R\Lambda_G(g)\longrightarrow R\Lambda_G^k(g)$ sending $q$ to
$q^{\frac{1}{k}}$. A good reference for $\lambda-$rings is Chapter
1 and 2, \cite{Yau10}.

\subsection{Quasi-elliptic cohomology}\label{orbqec} In this
section we introduce the definition of quasi-elliptic cohomology
$QEll^*_G$ in terms of orbifold K-theory, and then express it via
equivariant K-theory. We assume familiarity with
\cite{SegalequiK}. The reader may read Chapter 3 in \cite{ALRuan}
and \cite{Moe02} for a reference of orbifold K-theory.

When $G$ is finite, quasi-elliptic cohomology is defined from the
ghost loops in Definition \ref{ghostloopdef}. By proposition
\ref{skullofghost} and Example \ref{finiteghost}, we can see the
groupoid $GhLoop(X/\!\!/G)$ is equivalent to the disjoint union of
some translation groupoids.  Before describing this equivalent
groupoid $\Lambda(X/\!\!/G)$ in detail, we recall what inertia
groupoid is. A reference for that is Section 4, \cite{LU01}.


\begin{definition}Let $\mathbb{G}$ be a groupoid. The inertia
groupoid $I(\mathbb{G})$ of $\mathbb{G}$ is defined as follows.

An object $a$ is an arrow in $\mathbb{G}$ such that its source and
target are equal. 
A morphism $v$ joining two objects $a$ and $b$ is an arrow $v$ in
$\mathbb{G}$ such that $$v\circ a=b\circ v.$$ In other words, $b$
is the conjugate of $a$ by $v$, $b=v\circ a\circ v^{-1}$.

\label{inertiagroupoid}
\end{definition}

Let  $X$ a $G-$space.

\begin{example} The
inertia groupoid $I(X/\!\!/G)$ 
is the groupoid with

\textbf{objects}: the space $\coprod\limits_{g\in G}X^{g}$

\textbf{morphisms}: the space $\coprod\limits_{g, g'\in
G}C_G(g,g')\times X^g$ where $C_G(g,g')=\{\sigma\in
G|g'\sigma=\sigma g\}\subseteq G.$

For $x\in X^g$ and $(\sigma, g)\in C_G(g,g')\times X^g$, $(\sigma,
g)(x)=\sigma x\in X^{g'}.$ \label{torsionquotient}
\end{example}


\begin{definition}The groupoid $\Lambda(X/\!\!/G)$ has the same objects as
$I(X/\!\!/G)$ but richer morphisms
$$\coprod\limits_{g, g'\in G}\Lambda_G(g, g')\times X^g$$
where $\Lambda_G(g, g')$ is the quotient of $C_G(g, g')\times
\mathbb{R}$ under the equivalence $$(x, t)\sim (gx, t-1)=(xg',
t-1).$$ For an object $x\in X^g$ and a morphism $([\sigma, t],
g)\in \Lambda_G(g, g')\times X^g$, $([\sigma, t], g)(x)=\sigma
x\in X^{g'}.$ The composition of the morphisms is defined by
\begin{equation}[\sigma_1, t_1][\sigma_2,
t_2]=[\sigma_1\sigma_2, t_1+t_2].\end{equation}
\label{lambdaoidef}\end{definition}

\begin{definition}
The quasi-elliptic cohomology $QEll^*_G(X)$ is defined to be
$K^*_{orb}(GhLoop(X/\!\!/G))\cong K^*_{orb}(\Lambda(X/\!\!/G))$.
\label{qecdef}\end{definition} We can unravel the definition and
express it via equivariant K-theory.

Let $\sigma\in G$. The fixed point space $X^{\sigma}$ is a
$C_G(\sigma)-$space. We can define a $\Lambda_G(\sigma)-$action on
$X^{\sigma}$ by
$$[g, t]\cdot x:=g\cdot x.$$
Then we have
\begin{proposition}
\begin{equation}QEll^*_G(X)=\prod_{g\in
G_{conj}}K^*_{\Lambda_G(g)}(X^{g})=\bigg(\prod_{g\in
G}K^*_{\Lambda_G(g)}(X^{g})\bigg)^G,\end{equation}  where
$G_{conj}$ is a set of representatives of $G-$conjugacy classes in
$G$.
\end{proposition}
Thus, for each $g\in\Lambda_G(g)$, we can define the projection
$$\pi_g: QEll^*_G(X)\longrightarrow K^*_{\Lambda_G(g)}(X^{g}).$$
For the singe point space, we have
\begin{equation}QEll^0_G(\mbox{pt})\cong\prod_{g\in G_{conj}} R\Lambda_G(g). \label{qecpt}\end{equation}

We have the ring homomorphism
$$\mathbb{Z}[q^{\pm}]=K^0_{\mathbb{T}}(\mbox{pt})\buildrel{\pi^*}\over\longrightarrow K^0_{\Lambda_G(g)}(\mbox{pt})\longrightarrow
K^0_{\Lambda_G(g)}(X)$$ where $\pi: \Lambda_G(g)\longrightarrow
\mathbb{T}$ is the projection defined in (\ref{pizq}) and the
second is via the collapsing map $X\longrightarrow \mbox{pt}$. So
$QEll_G^*(X)$ is naturally a
$\mathbb{Z}[q^{\pm}]-$algebra. 

\subsection{Properties}\label{propertiesqec} In this section
 we discuss some properties of $QEll^*_G$,
including the restriction map, the K\"{u}nneth map on it, its
tensor product and the change-of-group isomorphism.

Since each homomorphism $\phi: G\longrightarrow H$ induces a
well-defined homomorphism $\phi_{\Lambda}:
\Lambda_G(\tau)\longrightarrow\Lambda_H(\phi(\tau))$ for each
$\tau$ in $G$, we can get the proposition below directly.
\begin{proposition}For each homomorphism $\phi: G\longrightarrow H$, it induces a ring map
$$\phi^*: QEll^*_H(X)\longrightarrow QEll^*_G(\phi^*X)$$ characterized by the commutative diagrams

\begin{equation}\begin{CD}QEll^*_H(X) @>{\phi^*}>> QEll^*_G(\phi^*X) \\ @V{\pi_{\phi(\tau)}}VV  @V{\pi_{\tau}}VV  \\
K^*_{\Lambda_H(\phi(\tau))}(X^{\phi(\tau)}) @>{\phi^*_{\Lambda}}>>
 K^*_{\Lambda_G(\tau)}(X^{\phi(\tau)})\end{CD}\end{equation} for any $\tau \in G$. So $QEll^*_G$ is functorial in $G$.
\label{restrictionq}\end{proposition}

Moreover, we can define K\"{u}nneth map of quasi-elliptic
cohomology induced from that on equivariant $K$-theory.

Let $G$ and $H$ be two finite groups. $X$ is a $G$-space and $Y$
is a $H$-space. Let $\sigma\in G$ and $\tau\in H$. Let
$\Lambda_G(\sigma)\times_{\mathbb{T}}\Lambda_H(\tau)$ denote the
fibered product of the morphisms
$$\Lambda_G(\sigma)\buildrel{\pi}\over\longrightarrow
\mathbb{T}\buildrel{\pi}\over\longleftarrow\Lambda_H(\tau).$$ It
is isomorphic to $\Lambda_{G\times H}(\sigma, \tau)$ under the
correspondence
$$([\alpha, t], [\beta, t])\mapsto [\alpha, \beta, t].$$

Consider the composition below
\begin{align*}T: &K_{\Lambda_G(\sigma)}(X^{\sigma})\otimes
K_{\Lambda_H(\tau)}(Y^{\tau})\longrightarrow
K_{\Lambda_{G}(\sigma)\times\Lambda_{H}(\tau)}(X^{\sigma}\times
Y^{\tau})\buildrel{res}\over\longrightarrow \\
&K_{\Lambda_{G}(\sigma)\times_{\mathbb{T}}\Lambda_{H}(\tau)}(X^{\sigma}\times
Y^{\tau}) \buildrel{\cong}\over\longrightarrow K_{\Lambda_{G\times
H}(\sigma, \tau)}((X\times Y)^{(\sigma, \tau)}).\end{align*} where
the first map is the K\"{u}nneth map of equivariant K-theory, the
second is the restriction map  and the third is the isomorphism
induced by the group isomorphism $\Lambda_{G\times H}(\sigma,
\tau)\cong\Lambda_G(\sigma)\times_{\mathbb{T}}\Lambda_H(\tau)$.

For any $g\in G$, let $1$ denote the trivial line bundle over
$X^g$ and let $q$ denote the line bundle $1\odot_{\mathbb{C}} q$
over $X^g$. The map $T$ above sends both $1\otimes q$ and
$q\otimes 1$ to $q$. So we get the well-defined map
\begin{equation}K^*_{\Lambda_G(\sigma)}(X^{\sigma})\otimes_{\mathbb{Z}[q^{\pm}]}K^*_{\Lambda_H(\tau)}(Y^{\tau})\longrightarrow
 K_{\Lambda_{G\times H}(\sigma, \tau)}((X\times
Y)^{(\sigma, \tau)}).\label{ku}\end{equation}

\begin{definition}The tensor produce of quasi-elliptic cohomology is defined by
\begin{equation}QEll^*_G(X)\otimes_{\mathbb{Z}[q^{\pm}]}QEll^*_H(Y)
\cong\prod_{\sigma\in G_{conj}\mbox{,   } \tau\in
H_{conj}}K^*_{\Lambda_G(\sigma)}(X^{\sigma})\otimes_{\mathbb{Z}[q^{\pm}]}K^*_{\Lambda_H(\tau)}(Y^{\tau}).\label{qectensor}\end{equation}
The direct product of the maps defined in (\ref{ku}) gives a ring
homomorphism
$$QEll^*_G(X)\otimes_{\mathbb{Z}[q^{\pm}]}QEll^*_H(Y)\longrightarrow
QEll^*_{G\times H}(X\times Y),$$ which is the K\"{u}nneth map of
quasi-elliptic cohomology.
\end{definition}

By Lemma \ref{cl} we have
$$QEll^*_G(\mbox{pt})\otimes_{\mathbb{Z}[q^{\pm}]}QEll^*_H(\mbox{pt})=QEll^*_{G\times H}(\mbox{pt}).$$
More generally, we have the proposition below.
\begin{proposition}

Let $X$ be a $G\times H-$space with trivial $H-$action and let
$\mbox{pt}$ be the single point space with trivial $H-$action.
 Then we have
$$QEll_{G\times H}(X)\cong QEll_G(X)\otimes_{\mathbb{Z}[q^{\pm}]} QEll_H(\mbox{pt}).$$

Especially, if $G$ acts trivially on $X$, we have
$$QEll_G(X)\cong   QEll(X)\otimes_{\mathbb{Z}[q^{\pm}]}
QEll_G(\mbox{pt}).$$ Here $QEll^*(X)$ is
$QEll^*_{\{e\}}(X)=K^*_{\mathbb{T}}(X)$.
\end{proposition}

\begin{proof}

\begin{align*}QEll_{G\times H}(X) &=\prod\limits_{\substack{g\in
G_{conj}
\\ h\in H_{conj}}}K_{\Lambda_{G\times H}(g, h)}(X^{(g,
h)})\cong \prod\limits_{\substack{g\in G_{conj} \\ h\in H_{conj}}
}K_{\Lambda_{G}(g)\times_{\mathbb{T}} \Lambda_{H}(h)}(X^{g})\\
&\cong \prod\limits_{\substack{g\in G_{conj} \\ h\in H_{conj}}}
K_{\Lambda_{G}(g)}(X^{g})\otimes_{\mathbb{Z}[q^{\pm}]}
K_{\Lambda_H(h)}(\mbox{pt})=
QEll_G(X)\otimes_{\mathbb{Z}[q^{\pm}]} QEll_H(\mbox{pt}).
\end{align*}

\end{proof}


\begin{proposition}
If $G$ acts freely on $X$,
$$QEll^*_G(X)\cong QEll^*_e(X/G).$$
\end{proposition}
\begin{proof}
Since $G$ acts freely on $X$, $$X^{\sigma}=\begin{cases}\emptyset,
&\text{if $\sigma\neq e$;}\\ X, &\text{if
$\sigma=e$.}\end{cases}$$ Thus,
$QEll^*_G(X)\cong\prod\limits_{\sigma\in
G_{conj}}K^*_{\Lambda_G(\sigma)/C_G(\sigma)}(X^{\sigma}/C_G(\sigma))\cong
K^*_{\mathbb{T}}(X/G).$

Since $\mathbb{T}$ acts trivially on $X$, we have
$K^*_{\mathbb{T}}(X/G)=QEll^*_e(X/G)$ by definition. It is
isomorphic to $K^*(X/G)\otimes R\mathbb{T}$.\end{proof}


We also have the change-of-group isomorphism as  in equivariant
$K$-theory.

Let $H$ be a subgroup of $G$ and $X$ a $H$-space. Let $\phi:
H\longrightarrow G$ denote the inclusion homomorphism. The
change-of-group map $\rho^G_H: QEll^*_G(G\times_HX)\longrightarrow
QEll^*_H(X)$ is defined as the composite
\begin{equation}\rho^G_H:   QEll^*_G(G\times_HX)\buildrel{\phi^*}\over\longrightarrow
QEll^*_H(G\times_H X)\buildrel{i^*}\over\longrightarrow
QEll_H^*(X)\label{changeofgroup}
\end{equation}
where $\phi^*$ is the restriction map and $i: X\longrightarrow
G\times_HX$ is the $H-$equivariant map defined by $i(x)=[e, x].$

\begin{proposition} The change-of-group map
$$\rho^G_H: QEll^*_G(G\times_H X)\longrightarrow
QEll^*_H(X)$$ defined in (\ref{changeofgroup}) is an
isomorphism.\end{proposition}
\begin{proof}
For any $\tau\in H_{conj}$, there exists a unique
$\sigma_{\tau}\in G_{conj}$ such that
$\tau=g_{\tau}\sigma_{\tau}g_{\tau}^{-1}$ for some $g_{\tau}\in
G$.  Consider the maps \begin{equation}\begin{CD}
\Lambda_G(\tau)\times_{\Lambda_H(\tau)}X^{\tau}@>{[[a, t], x
]\mapsto [a, x]}>> (G\times_H X)^{\tau}@>{[u, x]\mapsto
[g_{\tau}^{-1}u, x]}>>
(G\times_HX)^{\sigma}.\end{CD}\end{equation} The first map is
$\Lambda_G(\tau)-$equivariant and the second is equivariant with
respect to the homomorphism $c_{g_{\tau}}:
\Lambda_{G}(\sigma)\longrightarrow \Lambda_G(\tau)$ sending $[u,
t]\mapsto [g_{\tau} u g_{\tau}^{-1}, t]$. Taking a coproduct over
all the elements $\tau\in H_{conj}$ that are conjugate to
$\sigma\in G_{conj}$ in $G$, we get an isomorphism
$$\gamma_{\sigma}: \coprod_{\tau}\Lambda_G(\tau)\times_{\Lambda_H(\tau)} X^{\tau}\longrightarrow
(G\times_HX)^{\sigma}$$ which is $\Lambda_G(\sigma)-$equivariant
with respect to $c_{g_{\tau}}$. Then we have the map
\begin{equation}\gamma:=\prod_{\sigma\in G_{conj}}\gamma_{\sigma}:
\prod_{\sigma\in G_{conj}}
K^*_{\Lambda_G(\sigma)}(G\times_HX)^{\sigma}\longrightarrow
\prod_{\sigma\in
G_{conj}}K^*_{\Lambda_G(\sigma)}(\coprod_{\tau}\Lambda_G(\tau)\times_{\Lambda_H(\tau)}
X^{\tau})
\end{equation}

It is straightforward to check the change-of-group map coincide
with the composite \begin{align*} QEll^*_{G}(G\times_H
X)\buildrel{\gamma}\over\longrightarrow \prod_{\sigma\in
G_{conj}}K^*_{\Lambda_G(\sigma)}(\coprod_{\tau}\Lambda_G(\tau)\times_{\Lambda_H(\tau)}
X^{\tau})\longrightarrow &\prod_{\tau\in
H_{conj}}K^*_{\Lambda_H(\tau)}(X^{\tau})\\&=QEll^*_{H}(X)\end{align*}
with  the second map  the change-of-group isomorphism in
equivariant $K-$theory.
\end{proof}


\section{Power Operation}\label{poweroperation}


In Section \ref{s2complete} we define power operations for
equivariant quasi-elliptic cohomology $QEll_G^*(-)$. We show in
Theorem \ref{main1p} that they satisfy the axioms that Ganter
established in Definition 4.3, \cite{Gan06} for equivariant power
operations.

The power operation of quasi-elliptic cohomology is of the form
\begin{align*}\mathbb{P}_n=&\prod_{(\underline{g}, \sigma)\in
(G\wr\Sigma_n)_{conj}}\mathbb{P}_{(\underline{g},\sigma)}:&
\\
&QEll^*_G(X)\longrightarrow QEll^*_{G\wr\Sigma_n}(X^{\times n})
=&\prod_{(\underline{g}, \sigma)\in
(G\wr\Sigma_n)_{conj}}K_{\Lambda_{G\wr\Sigma_n}(\underline{g},
\sigma)}((X^{\times n})^{(\underline{g}, \sigma)}),\end{align*}
where $\mathbb{P}_n$ maps a bundle over the groupoid
$$\Lambda(X/\!\!/G)$$ to a bundle over
$$\Lambda(X^{\times n}/\!\!/(G\wr\Sigma_n)),$$ and each
$\mathbb{P}_{(\underline{g},\sigma)}$ maps a bundle over
$$\Lambda(X/\!\!/G)$$ to a $\Lambda_{G\wr\Sigma_n}(\underline{g}, \sigma)-$bundle over the space $(X^{\times
n})^{(\underline{g},
\sigma)}/\!\!/\Lambda_{G\wr\Sigma_n}(\underline{g}, \sigma).$


\bigskip



We construct each $\mathbb{P}_{(\underline{g}, \sigma)}$ as the
composition below.
\begin{align}QEll^*_G(X)&\buildrel{U^*}\over\longrightarrow
K^*_{orb}(\Lambda^1_{(\underline{g}, \sigma)}(X))
\buildrel{(\mbox{ })_k^{\Lambda}}\over\longrightarrow
K^*_{orb}(\Lambda^{var}_{(\underline{g}, \sigma)}(X)) \label{pgs1}\\
&\buildrel{\boxtimes}\over\longrightarrow
K^*_{orb}(d_{(\underline{g},
\sigma)}(X))\buildrel{f^*_{(\underline{g},
\sigma)}}\over\longrightarrow
K^*_{\Lambda_{G\wr\Sigma_n}(\underline{g}, \sigma)}((X^{\times
n})^{(\underline{g}, \sigma)}),\label{pgs2}\end{align}
 where
$k\in\mathbb{Z}$ and $(i_1, \cdots i_k)$ goes over all the
$k-$cycles of $\sigma$. We explain the first three functors in
detail in Section \ref{s2complete}. In Section \ref{111} we
construct the isomorphism $f_{(\underline{g}, \sigma)}$ between
the groupoid $$\Lambda(X^{\times n}/\!\!/(G\wr\Sigma_n))$$ and the
groupoid $d((X/\!\!/G)\wr\Sigma_n)$ constructed in Definition
\ref{dxgsigman}. With it, it is convenient to construct the
explicit formula of the power operation.

\subsection{Loop Space of Symmetric Power}\label{111}


\subsubsection{The groupoid $d((X/\!\!/G)\wr\Sigma_n)$}\label{setupad}



For an introduction of actions of wreath product $G\wr\Sigma_n$ on
$X^{\times n}$ and symmetric power $\mathbb{G}\wr\Sigma_n$ of a
groupoid $\mathbb{G}$, we refer the readers to Section 4.1,
\cite{Gan07}. The symmetric power $(X/\!\!/G)\wr\Sigma_n$ is
isomorphic to $X^{\times n}/\!\!/(G\wr\Sigma_n)$.

Before introducing the groupoid $d((X/\!\!/G)\wr\Sigma_n)$, we
need to introduce several ingredients.

\begin{definition}[$\Lambda^k(X/\!\!/G)$] The groupoid $\Lambda^k(X/\!\!/G)$ has the same objects as
$\Lambda(X/\!\!/G)$ but different morphisms
$$\coprod\limits_{g, g'\in G}\Lambda^k_G(g, g')\times X^g$$
where $\Lambda^k_G(g, g')$ is the quotient of $C_G(g, g')\times
\mathbb{R}$ under the equivalence $$(x, t)\sim (gx, t-k)=(xg',
t-k).$$ For an object $x\in X^g$ and a morphism $([\sigma, t],
g)\in \Lambda^k_G(g, g')\times X^g$, $([\sigma, t], g)(x)=\sigma
x\in X^{g'}.$ The composition of the morphisms is defined by
\begin{equation}[\sigma_1, t_1][\sigma_2,
t_2]=[\sigma_1\sigma_2, t_1+t_2].\end{equation}
\label{lambdaoidefk}\end{definition}

\begin{definition}[Fibred wreath product] The groupoid
$\Lambda^k(X/\!\!/G)\wr_{\mathbb{T}}\Sigma_N$ is defined to be the
subgroupoid of the symmetric power
$\Lambda^k(X/\!\!/G)\wr\Sigma_N$ with the same objects but only
those morphisms
$$(([b_1, t_1], \cdots
[b_{N}, t_N],    \tau), x)$$ with all the $t_j$s having the same
image under the quotient map
$\mathbb{R}/k{\mathbb{Z}}\longrightarrow\mathbb{R}/\mathbb{Z}$.

The isotropy group of each object in $\prod\limits_1^N X^g$ is
$\Lambda_G^k(g)\wr_{\mathbb{T}}\Sigma_N$. \end{definition}

Let $Y$ be an $H-$space.
\begin{definition}[Fibred product and fibred coproduct]
The groupoid
$$\big(\Lambda^{k_1}(X/\!\!/G)\wr_{\mathbb{T}}\Sigma_{N_1}\big)\times_{\mathbb{T}}\big(\Lambda^{k_2}(Y/\!\!/H)\wr_{\mathbb{T}}\Sigma_{N_2}\big)$$
is defined to be the subgroupoid of
$\Lambda^{k_1}(X/\!\!/G)\wr_{\mathbb{T}}\Sigma_{N_1}\times
\Lambda^{k_2}(Y/\!\!/H)\wr_{\mathbb{T}}\Sigma_{N_2}$ with the same
objects but only those morphisms $$\big((([g_{1}, t_{1, 1}],
\cdots [g_{ N_1}, t_{1, N_1}],    \sigma_1), x), (([h_1, t_{2,
1}], \cdots [h_{N_2}, t_{2, N_2}],    \sigma_2), y)\big)$$ with
all the $t_{i, j_i}$s having the same image under the quotient map
$\mathbb{R}/k_i{\mathbb{Z}}\longrightarrow\mathbb{R}/\mathbb{Z}$,
for $i=1, 2$ and $j_i=1, \cdots N_i$.

The isotropy group of each object in $\prod\limits_1^{N_1}
X^g\prod\limits_1^{N_2} Y^h$ is
$$\big(\Lambda_G^{k_1}(g)\wr_{\mathbb{T}}\Sigma_{N_1}\big)\times_{\mathbb{T}}
\big(\Lambda_H^{k_2}(h)\wr_{\mathbb{T}}\Sigma_{N_2}\big).$$

\bigskip
We can define the fibred coproduct
$\big(\Lambda^{k_1}(X/\!\!/G)\wr_{\mathbb{T}}\Sigma_{N_1}\big)\coprod_{\mathbb{T}}\big(\Lambda^{k_2}(Y/\!\!/H)\wr_{\mathbb{T}}\Sigma_{N_2}\big)$
in the same way.
\end{definition}

Let $\sigma\in\Sigma_n$ correspond to the partition $n =
\sum\limits_kkN_k$, i.e. it has $N_k$ $k-$cycles. Assume that for
each cycle $(i_1, \cdots i_k)$ of $\sigma$,  $i_1< i_2\cdots <
i_k$.

For $(\underline{g}, \sigma)\in G\wr\Sigma_n$, we consider the
orbits of the bundle
$G\times\underline{\underline{n}}\longrightarrow\underline{\underline{n}}$
under the action by $(\underline{g}, \sigma)$. The orbits of
$\underline{\underline{n}}$ under the action by $\sigma$
corresponds to the cycles in the cycle decomposition of $\sigma$.
The bundle
$G\times\underline{\underline{n}}\longrightarrow\underline{\underline{n}}$
is the disjoint union of the $G-$bundles

$$\bigsqcup_{(i_1 \cdots i_k)} (G\times \{i_1, \cdots i_k\}\longrightarrow \{i_1, \cdots i_k\})$$ where $(i_1, \cdots i_k)$ goes over all the
cycles of $\sigma$. Each bundle $G\times \{i_1, \cdots
i_k\}\longrightarrow \{i_1, \cdots i_k\}$ is an orbit of
$G\times\underline{\underline{n}}\longrightarrow\underline{\underline{n}}$
under the action by $(\underline{g}, \sigma)$.

Let $C_G(g, g')$ denote $\{x\in G| gx=xg'\}.$ Two $G-$subbundles
$$G\times \{i_1, \cdots i_k\}\longrightarrow \{i_1, \cdots
i_k\}\mbox{    and }G\times \{j_1, \cdots j_m\}\longrightarrow
\{j_1, \cdots j_m\}$$ are $(\underline{g}, \sigma)-$equivariant
equivalent if and only if $k=m$ and $C_G(g_{i_k}\cdots g_{i_1},
g_{j_k}\cdots g_{j_1})$ is nonempty. For each $k$-cycle $i=(i_1,
\cdots i_k)$ of $\sigma$, let $W^{\sigma}_i$ denote the set of all
the $G-$subbundles $G\times \{j_1, \cdots j_m\}\longrightarrow
\{j_1, \cdots j_m\}$ that are $(\underline{g},
\sigma)-$isomorphic to $G\times \{i_1, \cdots i_k\}\longrightarrow
\{i_1, \cdots i_k\}$. There is a bijection between $W^{\sigma}_i$
and the set $$\{j=(j_1, \cdots j_k)\mbox{ } | \mbox{  }(j_1,
\cdots j_k)\mbox{ is a k-cycle of }\sigma\mbox{ and }
C_G(g_{i_k}\cdots g_{i_1}, g_{j_k}\cdots g_{j_1})\mbox{ is
nonempty}\}.$$ Let $M^{\sigma}_{i}$ denote the size of the set
$W^{\sigma}_i$.  Let $\alpha^i_1, \alpha^i_2, \cdots
\alpha^i_{M^{\sigma}_i}$ denote all the elements of the set
$W^{\sigma}_i$. Obviously, $i=(i_1, \cdots i_k)$ is in
$W^{\sigma}_i$. So we can assume it is $\alpha^i_1$.

For any $k-$cycle $i$ and $m-$cycle $j$ of $\sigma$, if $k=m$ and
$C_G(g_{i_k}\cdots g_{i_1}, g_{j_k}\cdots g_{j_1})$ is nonempty,
$W^{\sigma}_i$ and $W^{\sigma}_j$ are the same set. Otherwise,
they are disjoint. The set of all the $k-$cycles of $\sigma$ can
be divided into the disjoint union of several $W^{\sigma}_i$s. We
can pick a set of representatives $\theta_k$ of $k-$cycles of
$\sigma$ such that the set of $k-$cycles of $\sigma$ equals the
disjoint union
$$\coprod_{i\in\theta_k}W^{\sigma}_i.$$

\begin{definition}[$d_{(\underline{g}, \sigma)}(X)$]
The groupoid $d_{(\underline{g}, \sigma)}(X)$ is defined to be a
full subgroupoid of
$\prod\limits_{k}\!{_{\mathbb{T}}}\prod\limits_{i\in\theta_k}\!{_{\mathbb{T}}}\Lambda^k(X/\!\!/G)\wr_{\mathbb{T}}\Sigma_{M^{\sigma}_i}$
with objects the points of the space
$$\prod_k\prod_{(i_1, \cdots
i_k)}X^{g_{i_k}\cdots g_{i_1}},$$ where the second product goes
over all the $k-$cycles of $\sigma$. \label{ADgroupoid}
\end{definition}

\begin{definition}[$d((X/\!\!/G)\wr\Sigma_n)$] The groupoid
$d((X/\!\!/G)\wr\Sigma_n)$ is defined to be
$$\coprod\limits_{(\underline{g},
\sigma)}\!{_{\mathbb{T}}}\mbox{   }d_{(\underline{g},
\sigma)}(X)$$ where $(\underline{g}, \sigma)$ goes over
$(G\wr\Sigma_n)_{conj}$. \label{dxgsigman}
\end{definition}

\begin{proposition}
Each $d_{(\underline{g}, \sigma)}(X)$ is isomorphic to the
translation groupoid $$\big(\prod_k\prod_{(i_1, \cdots
i_k)}X^{g_{i_k}\cdots
g_{i_1}}\big)/\!\!/\big(\prod\limits_{k}\!{_{\mathbb{T}}}\prod\limits_{j\in\theta_k}\!{_{\mathbb{T}}}\Lambda_G^k(\alpha_j)\wr_{\mathbb{T}}\Sigma_{M^{\sigma}_j}\big)$$
where $\alpha_j=g_{j_k}\cdots g_{j_1}$ with $j=(j_1, \cdots j_k)$.
\label{trans}\end{proposition} The proof is straightforward.

To study $K_{orb}(d_{(\underline{g}, \sigma)}(X))$, we start by
studying the representation ring of the wreath product
$$\prod\limits_{k}\prod\limits_{j\in\theta_k}\Lambda_G^k(\alpha_j)\wr\Sigma_{M^{\sigma}_j}.$$ Theorem \ref{repwr} gives all
the irreducible representations of a wreath product. It is Theorem
Theorem 4.3.34 in \cite{JK81}.

\begin{theorem}Let $\{\rho_k\}^N_1$ be a complete family of irreducible representations of $G$ and let $V_k$ be the corresponding representation space for $\rho_k$. Let $(n)$ be a partition
of $n$. $(n)=(n_1, \cdots n_N).$ Let $D_{(n)}$ be the
representation $$\rho_1^{\otimes
n_1}\otimes\cdots\otimes\rho_N^{\otimes n_N}$$ of $G^{\times N}$
on $V_1^{\otimes n_1}\otimes\cdots\otimes V_N^{\otimes n_N}.$ Let
$\Sigma_{(n)}=\Sigma_{n_1}\times\cdots\times\Sigma_{n_N}.$

Let $(D_{(n)})^{\sim}$ be the extension of $D_{(n)}$ from
$G^{\times n}$ to $G\wr\Sigma_{(n)}$ defined by
\begin{align*}&(D_{(n)})^{\sim}((g_{1, 1},\cdots g_{1, n_1}, \cdots g_{N, 1}, \cdots g_{N, n_N}; \sigma))\\ &(v_{1,
1}\otimes\cdots\otimes v_{1, n_1}\otimes\cdots\otimes
v_{N,1}\otimes\cdots\otimes v_{N,
n_N})\\=&\bigotimes^{N}_{k=1}\rho_k(g_{k,
1})v_{k,\sigma_k^{-1}(1)}\otimes\cdots\otimes\rho_k(g_{k,
n_k})v_{k,\sigma_k^{-1}(n_k)},\end{align*} where
$\sigma=\sigma_1\times\cdots\times\sigma_N$ with each
$\sigma_k\in\Sigma_{n_k}$.

Let $D_{\tau}$ with $\tau\in R\Sigma_{(n)}$ be the representation
of $G\wr\Sigma_{(n)}$ defined by \begin{equation}D_{\tau}((g_{1,
n_1},\cdots g_{N,n_N};
\sigma)):=\tau(\sigma).\label{drepsymm}\end{equation}

Then,
\begin{align*}\{Ind|^{G\wr\Sigma_n}_{G\wr\Sigma_{(n)}}(D_{(n)})^{\sim}\otimes
D_{\tau}| &(n)=(n_1, \cdots n_N)\mbox{  goes over all the partitions};\\
&\tau\mbox{    goes over all the irreducible representations of }
\Sigma_{(n)}.\}\end{align*}goes over all the irreducible
representations  of $G\wr\Sigma_{n}$
nonrepeatedly.\label{repwr}\end{theorem}

The proof of Theorem \ref{repfibwr} is analogous to that of
Theorem \ref{repwr} in \cite{JK81}, applying Clifford's theory in
\cite{Clifford1937} and \cite{CurtisReiner62}. Note that
$$\{\rho_1\otimes_{\mathbb{Z}[q^{\pm}]}\cdots
\otimes_{\mathbb{Z}[q^{\pm}]} \rho_n\mbox{    }|\mbox{   Each
}\rho_j\mbox{    is    an irreducible representation of     }
\Lambda_G(\sigma).\}$$ goes over all the irreducible
representations of the fibred product
$$\Lambda_G(\sigma)\times_{\mathbb{T}}\cdots \times_{\mathbb{T}}
\Lambda_G(\sigma).$$

\begin{theorem}Let $\{\rho_k\}^N_1$ be a basis of the $\mathbb{Z}[q^{\pm}]-$module $R\Lambda_G(\sigma)$ and let $V_k$ be the corresponding representation space for $\rho_k$.
Let $(n)$ be a partition of $n$. $(n)=(n_1, \cdots n_N).$ Let
$D^{\mathbb{T}}_{(n)}$ be the
$\Lambda_G(\sigma)^{\times_{\mathbb{T}} n}-$representation
$$\rho_1^{\otimes_{\mathbb{Z}[q^{\pm}]}
n_1}\otimes_{\mathbb{Z}[q^{\pm}]}\cdots\otimes_{\mathbb{Z}[q^{\pm}]}\rho_N^{\otimes_{\mathbb{Z}[q^{\pm}]}
n_N}$$ on the space $V_1^{\otimes n_1}\otimes\cdots\otimes
V_N^{\otimes n_N}.$ Let
$\Sigma_{(n)}=\Sigma_{n_1}\times\cdots\times\Sigma_{n_N}.$

Let $(D^{\mathbb{T}}_{(n)})^{\sim}$ be the extension of $D_{(n)}$
from $\Lambda_G(\sigma)^{\times_{\mathbb{T}} n}$ to
$\Lambda_G(\sigma)\wr_{\mathbb{T}}\Sigma_{(n)}$ defined by
\begin{align*}&(D_{(n)}^{\mathbb{T}})^{\sim}(([g_{1, 1}, t],\cdots [g_{1, n_1},  t], \cdots [g_{N, 1},  t], \cdots [g_{N, n_N},   t]; \sigma))
\\ &(v_{1, 1}\otimes\cdots\otimes v_{1, n_1}\otimes\cdots\otimes
v_{N,1}\otimes\cdots\otimes v_{N,
n_N})\\=&{\bigotimes\nolimits_{\mathbb{Z}[q^{\pm}]}}\rho_k([g_{k,
1},
t])v_{k,\sigma_k^{-1}(1)}\otimes_{\mathbb{Z}[q^{\pm}]}\cdots\otimes_{\mathbb{Z}[q^{\pm}]}\rho_k([g_{k,
n_k},   t])v_{k,\sigma_k^{-1}(n_k)},\end{align*} where $k$ is from
$1$ to $N$ and $\sigma=\sigma_1\times\cdots\times\sigma_N$ with
each $\sigma_k\in\Sigma_{n_k}$.

Let $D^{\mathbb{T}}_{\tau}$ with $\tau\in R\Sigma_{(n)}$ be the
representation of $\Lambda_G(\sigma)\wr_{\mathbb{T}}\Sigma_{(n)}$
defined by \begin{equation}D^{\mathbb{T}}_{\tau}(([g_{1, n_1},
t],\cdots [g_{N,n_N},   t];
\sigma)):=\tau(\sigma).\label{drepsymm}\end{equation}

Then,
\begin{align*}\{Ind|^{\Lambda_G(\sigma)\wr_{\mathbb{T}}\Sigma_n}_{\Lambda_G(\sigma)\wr_{\mathbb{T}}\Sigma_{(n)}}&(D^{\mathbb{T}}_{(n)})^{\sim}\otimes
D^{\mathbb{T}}_{\tau}\mbox{   }|\mbox{    } (n)=(n_1, \cdots n_N)\mbox{  goes over all the partitions};\\
&\tau\mbox{    goes over all the irreducible representations of }
\Sigma_{(n)}.\}\end{align*}goes over all the irreducible
representation nonrepeatedly of
$\Lambda_G(\sigma)\wr_{\mathbb{T}}\Sigma_{n}$.\label{repfibwr}\end{theorem}

From Theorem \ref{repwr}, the representation ring of each
$\Lambda_G^k(\alpha_j)\wr\Sigma_{M^{\sigma}_j}$ is a
$\mathbb{Z}[q^{\pm\frac{1}{k}}]-$module. Thus, the representation
ring of each $\Lambda_G^k(\alpha_j)\wr\Sigma_{M^{\sigma}_j}$ is a
$\mathbb{Z}[q^{\pm}]-$module via the
map
$$\mathbb{Z}[q^{\pm}]\longrightarrow\mathbb{Z}[q^{\pm\frac{1}{k}}]\mbox{,
}q\mapsto q^{\pm\frac{1}{k}}.$$ The representation ring
$$R(\prod_{k}\prod_{j\in\theta_k}\Lambda_G^k(\alpha_j)\wr\Sigma_{M^{\sigma}_j})\cong
\bigotimes_{k}\bigotimes_{j\in\theta_k}R(\Lambda_G^k(\alpha_j)\wr\Sigma_{M^{\sigma}_j})$$
is a $\mathbb{Z}[q^{\pm}]-$module. So is
$R(\prod\limits_{k}\!{_{\mathbb{T}}}\prod\limits_{j\in\theta_k}\!{_{\mathbb{T}}}\Lambda_G^k(\alpha_j)\wr_{\mathbb{T}}\Sigma_{M^{\sigma}_j}).$

Moreover, $K_{orb}(d_{(\underline{g}, \sigma)}(X))$ is a
$\mathbb{Z}[q^{\pm}]-$module via the map
\begin{equation}R(\prod\limits_{k}\!{_{\mathbb{T}}}\prod\limits_{j\in\theta_k}\!{_{\mathbb{T}}}\Lambda_G^k(\alpha_j)\wr_{\mathbb{T}}\Sigma_{M^{\sigma}_j})\cong
K_{orb}^0(d_{(\underline{g}, \sigma)}(\mbox{pt}))\longrightarrow
K_{orb}^0(d_{(\underline{g}, \sigma)}(X)),\end{equation} which is
induced by $X\longrightarrow\mbox{pt}$.

\subsubsection{The isomorphism $f_{(\underline{g}, \sigma)}$}\label{FunctorF}

Before we show in Theorem \ref{iprm} that the groupoids
$\Lambda(X^{\times n}/\!\!/(G\wr\Sigma_n))$ and
$d((X/\!\!/G)\wr\Sigma_n)$ are isomorphic,  we recall some
properties of $C_{G\wr\Sigma_n}((\underline{g}, \sigma),
(\underline{g}', \sigma'))$.

$(\underline{h}, \tau)$ is in $C_{G\wr \Sigma_n}
((\underline{g},\sigma), (\underline{g}',\sigma'))$ if and only if
$\tau\sigma'=\sigma\tau$ and
$g_{\sigma(\tau(i))}h_{\tau(i)}=h_{\tau(\sigma'(i))}g'_{\sigma'(i)},$
$\forall i.$ We can reinterpret these two conditions. Since
$\tau\in C_{\Sigma_n}(\sigma, \sigma')$, $\tau$ maps a $k$-cycle
$i=(i_1,\cdots i_k)$ of $\sigma'$ to a $k$-cycle $j=(j_1, \cdots
j_k)$ of $\sigma$. $\tau$ will still used to denote its map on the
cycles, such as $\tau(r)=s$. For each
$l\in\mathbb{Z}/k\mathbb{Z}$, let $\tau(i_l)=j_{l+m_i}$ where
$m_i$ depends only on $\tau$ and the cycle $i.$ Then, the second
conditions can be expressed as
\begin{equation}\forall l\in\mathbb{Z}/k\mathbb{Z},
g_{j_l}h_{j_{l-1}}=h_{j_l}g'_{i_{l-m_i}}.\label{rela2}\end{equation}

From this equivalence, we can induce that the element
$$h_{j_k}g'^{-1}_{i_{1-m_i}}\cdots
g'^{-1}_{i_{k-1}}g'^{-1}_{i_k}=g^{-1}_{j_1}\cdots g^{-1}_{j_{m_i}}
h_{j_{m_i}}$$ maps $g_{j_k}\cdots g_{j_1}$ to $g'_{i_k}\cdots
g'_{i_1}$ by conjugation. In other words,
\begin{equation}\beta^{\underline{h}, \tau}_{j,
i}:=h_{j_k}g'^{-1}_{i_{1-m_i}}\cdots
g'^{-1}_{i_{k-1}}g'^{-1}_{i_k}\label{betadef}\end{equation} is an
element in $C_G (g_{j_k}\cdots g_{j_1}, g'_{i_k}\cdots g'_{i_1})$.
Thus, $C_G (g_{j_k}\cdots g_{j_1}, g'_{i_k}\cdots g'_{i_1})$ is
nonempty.

First we show each component $(X^{\times n})^{(\underline{g},
\sigma)}/\!\!/\Lambda_{G\wr\Sigma_n}(\underline{g}, \sigma)$ is
isomorphic to the groupoid $d_{(\underline{g}, \sigma)}(X)$. We
construct a functor $$f_{(\underline{g}, \sigma)}: (X^{\times
n})^{(\underline{g},
\sigma)}/\!\!/\Lambda_{G\wr\Sigma_n}(\underline{g}, \sigma)
\longrightarrow d_{(\underline{g}, \sigma)}(X).$$ It sends a point
$$x=(x_1, \cdots x_n)\in
(X^{\times n})^{(\underline{g}, \sigma)}$$ to
$$\prod_k\prod_{(i_1, \cdots
i_k)}x_{i_k}.$$ Note that $x_{i_k}=x_{i_1}g_{i_1}=\cdots
=x_{i_{k-1}}g_{i_{k-1}}\cdots g_{i_1}.$

Let $[(\underline{h}, \tau), t]\in
\Lambda_{G\wr\Sigma_n}(\underline{g}, \sigma)$. Let $\tau$ send
the $k-$cycle $i=(i_1, \cdots i_k)$ of $\sigma$ to a $k-$cycle
$j=(j_1, \cdots j_k)$ of $\sigma$ and $\tau(i_1)=j_{1+m_i}$. We
have
$$f_{(\underline{g}, \sigma)}( \gamma\cdot [(\underline{h}, \tau),
t_0])= \prod_k\prod_{(i_1, \cdots
i_k)}x_{j_{m_i}}h_{j_{m_i}}=\prod_k\prod_{(i_1, \cdots
i_k)}x_{j_k}\cdot\beta^{\underline{h}, \tau}_{j, i},$$ where
$\beta^{\underline{h}, \tau}_{j, i}$ is the symbol defined in
(\ref{betadef}). So $f_{(\underline{g}, \sigma)}$ maps the
morphism $[(\underline{h}, \tau), t]$ to
$$\times_k\times_{i\in\theta_k}([\beta^{\underline{h}, \tau}_{\tau(1), 1}, m_1+t], \cdots [\beta^{\underline{h}, \tau}_{\tau(M^{\sigma}_i), M^{\sigma}_i}, m_{M^{\sigma}_i}+t],
\tau|_{W^{\sigma}_i})$$ where $\tau|_{W^{\sigma}_i}$ denotes the
permutation induced by $\tau$ on the set
$W^{\sigma}_i=\{\alpha^i_1, \alpha^i_2, \cdots
\alpha^i_{M^{\sigma}_i}\}$, $\tau^{-1}(j)$ is short for $\tau^{-1}
(\alpha^i_j)$ and $\tau(j_l)=\tau(j)_{l+m_j}$.

It sends the identity map $[(1, \cdots, 1, \mbox{Id}), 0]$ to the
identity $$\times_k\times_{i\in\theta_k}([1, 0], \cdots [1, 0],
\mbox{Id}),$$ and preserves composition of morphisms. So it is
well-defined.

\begin{theorem} The two groupoids $(X^{\times n})^{(\underline{g},
\sigma)}/\!\!/\Lambda_{G\wr\Sigma_n}(\underline{g}, \sigma)$ and
$d_{(\underline{g}, \sigma)}(X)$ are isomorphic. Thus,
this isomorphism induces a $\Lambda_{G\wr\Sigma_n}(\underline{g},
\sigma)-$action on the space $$\prod_k\prod_{(i_1, \cdots
i_k)}X^{g_{i_k}\cdots g_{i_1}}.$$
 \label{iptm}
\end{theorem}
\begin{proof}

We construct the inverse functor $$J_{(\underline{g}, \sigma)}:
d_{(\underline{g}, \sigma)}(X)\longrightarrow (X^{\times
n})^{(\underline{g},
\sigma)}/\!\!/\Lambda_{G\wr\Sigma_n}(\underline{g}, \sigma)$$ of
$f_{(\underline{g}, \sigma)}$. For an object $
\times_k\times_{i\in\theta_k}\nu_{i, k}$ in $d_{(\underline{g},
\sigma)}(X)$, $J_{(\underline{g},
\sigma)}(\times_k\times_{i\in\theta_k}\nu_{i, k})=\{\nu_m\}_1^n$
 with
$\nu_{i_k}=\nu_{i, k}|_{[0, 1]}$ and $\nu_{i_s}(t):=\nu_{i,
k}(s+t)g_{i_1}^{-1}\cdots g_{i_s}^{-1}.$

Let
$$\prod_k\prod_{i\in\theta_k}((u^i_1,m'^i_1),(u^i_2,
m'^i_2),\cdots(u^i_{M^{\sigma}_{i}}, m'^i_{M^{\sigma}_{i}}),
\varrho^k_i)$$ be a morphism in $d_{(\underline{g}, \sigma)}(X)$.
Let $t$ be a representative of the image of $m'^i_1$ in
$\mathbb{R}/\mathbb{Z}$. Then, each $m^i_k:=m'^i_k-t$ is an
integer.

When we know how $\tau\in C_{\Sigma_n}(\sigma)$ permutes the
cycles of $\sigma$, whose information is contained in those
$\varrho^k_i\in\Sigma_{M^{\sigma}_{i}}$, and the numbers $m^i_1,
\cdots m^i_{M^{\sigma}_{i}},$  we can get a unique $\tau$.
Explicitly, for any number $j_r=1, 2 \cdots n$, if $j_r$ is in a
$k-$cycle $(j_1, \cdots j_k)$ of $\sigma$ and it is in the set
$W^{\sigma}_i$, then $\tau$ maps $j_r$ to
$\varrho^k_i(j)_{r+m^i_j}$, i.e. the $r+m^i_j$-th element in the
cycle $\varrho^k_i(j)$ of $\sigma$.

For any $a\in W^{\sigma}_i$, $\forall k$ and $i$, we want
$u^i_a=\beta^{\underline{h}, \tau}_{\tau(a), a}$ for some
$\underline{h}$. Thus,
\begin{equation}h_{\tau(a)_k}=u^i_a g_{a_k}\cdots
g_{a_{1-m^i_a}}.\end{equation} By (\ref{rela2}) we can get all the
other $h_{\tau(a)_j}$.

It can be checked straightforward that $J_{(\underline{g},
\sigma)}$ is a well-defined functor. It does not depend on the
choice of the representative $t$.

$J_{(\underline{g}, \sigma)}\circ f_{(\underline{g}, \sigma)}=
\mbox{Id}$; $f_{(\underline{g}, \sigma)}\circ J_{(\underline{g},
\sigma)}= \mbox{Id}$. So the conclusion is proved.
\end{proof}

Then by Proposition \ref{trans}, we get the main conclusion in
Section \ref{111}.
\begin{theorem}The two groupoids $\Lambda((X/\!\!/G)\wr\Sigma_n)$ and $d((X/\!\!/G)\wr\Sigma_n)$ are isomorphic.\label{iprm}\end{theorem}

The last conclusion in this section is some properties of the
functor $f_{(\underline{g}, \sigma)}$.

\begin{proposition}
(i) If $\sigma=(1)\in\Sigma_1$, the morphism $f_{(g, (1))}$ is the
identity map on $X^g/\!\!/\Lambda_{G}(g)$.

(ii) Let $(\underline{g}, \sigma)\in G\wr\Sigma_n$ and
$(\underline{h}, \tau)\in G\wr\Sigma_m$. The groupoids
$$(X^{\times n})^{(\underline{g},
\sigma)}/\!\!/\Lambda_{G\wr\Sigma_n}(\underline{g},
\sigma)\times_{\mathbb{T}}(X^{\times m})^{(\underline{h},
\tau)}/\!\!/\Lambda_{G\wr\Sigma_m}(\underline{h}, \tau)$$ and
$$(X^{\times (n+m)})^{(\underline{g}, \underline{h}
\sigma\tau)}/\!\!/\Lambda_{G\wr\Sigma_{n+m}}(\underline{g},
\underline{h}, \sigma\tau)
$$ are isomorphic.

(iii)$f_{(\underline{g}, \sigma)}$ preserves cartesian product of
loops. The following diagram of groupoids commutes.
$$\xymatrix{(X^{\times n})^{(\underline{g},
\sigma)}/\!\!/\Lambda_{\!G\wr\Sigma_n\!}(\underline{g},
\sigma)\!\!\times_{\mathbb{T}}\!\!(X^{\times m})^{(\underline{h},
\tau)}/\!\!/\Lambda_{\!G\wr\Sigma_m\!}(\underline{h},
\tau)\ar[r]^>>>>>>{\cong}
\ar[d]_{f_{(\underline{g}, \sigma)}\times f_{(\underline{h},
\tau)}} &(X^{\times (n+m)})^{(\underline{g}, \underline{h},
\sigma\tau)}/\!\!/\Lambda_{\!G\wr\Sigma_{n+m}\!}\!(\underline{g},\!
\underline{h}, \!\sigma\tau)\!\ar[d]^{
f_{(\underline{g}, \underline{h}, \sigma\tau)}}\\
d_{(\underline{g}, \sigma)}(X)\times_{\mathbb{T}}
d_{(\underline{h}, \tau)}(X)\ar[r]^>>>>>>>>>>>>>>>>>>>>>>{\cong}
&d_{(\underline{g}, \underline{h}, \sigma\tau)}(X)}$$ \label{sim}
\end{proposition}

\begin{proof}

(i) is indicated in the proof of Theorem \ref{iptm}.

(ii) We can define a functor $\Phi$ from $$(X^{\times
n})^{(\underline{g},
\sigma)}/\!\!/\Lambda_{G\wr\Sigma_n}(\underline{g},
\sigma)\times_{\mathbb{T}}(X^{\times m})^{(\underline{h},
\tau)}/\!\!/\Lambda_{G\wr\Sigma_m}(\underline{h}, \tau)$$ to
$(X^{\times (n+m)})^{(\underline{g}, \underline{h}
\sigma\tau)}/\!\!/\Lambda_{G\wr\Sigma_{n+m}}(\underline{g},
\underline{h}, \sigma\tau)$ sending an object $(x_1, x_2)$ to
$(x_1, x_2)$ and a morphism $([\alpha,t],[\beta, t])$ to $[\alpha,
\beta, t].$  It is straightforward to check $\Phi$ is an
isomorphism between the groupoids.

(iii) The proof is left to the readers.

\end{proof}

\subsection{Total Power Operation of
$QEll^*_G$}\label{s2complete}

In this section   we construct the total power operations
 for quasi-elliptic cohomology and give
its explicit formula in (\ref{ee}). We show in Theorem
\ref{main1p} that they satisfy the axioms that Ganter concluded in
Definition 4.3, \cite{Gan06} for equivariant power operation.

We explain each map in the formula (\ref{pgs1}) and (\ref{pgs2}).
The functor  $U: \Lambda^1_{(\underline{g},
\sigma)}(X)\longrightarrow \Lambda(X/\!\!/G)$ is defined in
(\ref{functorU}). The pullback $(\mbox{ })^{\Lambda}_k$ is defined
in (\ref{rescalelambda}). The external product $\boxtimes$ is
explained in (\ref{lambdaproductd}). The fourth is the pullback by
$f_{(\underline{g}, \sigma)}$.

\textbf{The Functor $U$}

 For each $(\underline{g},
\sigma)\in G\wr\Sigma_n$, $r\in\mathbb{Z}$,
let $\Lambda^r_{(\underline{g}, \sigma)}(X)$ denote the groupoid 
with objects
$$\coprod_k\coprod_{(i_1, \cdots i_k)}X^{g_{i_k}\cdots g_{i_1}}$$
where $(i_1, \cdots i_k)$ goes over all the $k-$cycles of
$\sigma$, and with morphisms
$$\coprod_k\coprod_{(i_1, \cdots i_k), (j_1, \cdots
j_k)}\Lambda^r_G(g_{i_k}\cdots g_{i_1}, g_{j_k}\cdots
g_{j_1})\times X^{g_{i_k}\cdots g_{i_1}},$$ where $(i_1, \cdots
i_k)$ and $(j_1, \cdots j_k)$ go over all the $k-$cycles of
$\sigma$ respectively. It may not be a subgroupoid of
$\Lambda^r(X/\!\!/G)$ because there may be cycles $(i_1, \cdots
i_k)$ and $(j_1, \cdots j_m)$
such that $$g_{i_k}\cdots g_{i_1}=g_{j_m}\cdots g_{j_1}.$$ 

Let \begin{equation}U:\Lambda^1_{(\underline{g},
\sigma)}(X)\longrightarrow \Lambda(X/\!\!/
G)\label{functorU}\end{equation} denote the functor sending $x$ in
the component $X^{g_{i_k}\cdots g_{i_1}}$ to the $x$ in the
component $X^{g_{i_k}\cdots g_{i_1}}$ of $\Lambda(X/\!\!/ G)$, and
send each morphism $$([h, t], x)\mbox{   in
}\Lambda_G(g_{i_k}\cdots g_{i_1}, g_{j_k}\cdots g_{j_1})\times
X^{g_{i_k}\cdots g_{i_1}}$$ to $$([h, t], x)\mbox{   in
  }\Lambda_G(g_{i_k}\cdots g_{i_1}, g_{j_k}\cdots g_{j_1})\times
X^{g_{i_k}\cdots g_{i_1}}.$$ In the case that  $g_{i_k}\cdots
g_{i_1}$ and $g_{j_k}\cdots g_{j_1}$ are equal, $([h, t], x)$ is
an arrow inside a single connected component.

\bigskip

\textbf{The Functor $(\mbox{                  })_k$}

For each integer $k$, there is a functor of groupoids $(\mbox{
})_k:  \Lambda^k(X/\!\!/G)\longrightarrow \Lambda(X/\!\!/G)$ 
sending an object $x$ to $x$ and  a morphism $([h, t_0], x)\mbox{
    to     }([h, \frac{t_0}{k}], x).$ The composition
$((\mbox{ })_k)_r=(\mbox{ })_{kr}.$

The functor $(\mbox{                  })_k$ gives a well-defined
map $$K_{orb}(\Lambda(X/\!\!/G))\longrightarrow K_{orb}(\Lambda^k
(X/\!\!/G))$$ by pullback of bundles.  We still use the symbol
$(\mbox{                  })_k$ to denote it when there is no
confusion.  For any $\Lambda(X/\!\!/G)-$vector bundle
$\mathcal{V}$,
 $S^1$ acts on $(\mathcal{V})_k$ via
\begin{align*}q^{\frac{1}{k}}: \mathbb{R}/k\mathbb{Z} &\longrightarrow U(1)\\ a&\mapsto e^{\frac{2\pi ia}{k}}.\end{align*}
If $\mathcal{V}$ has the decomposition
$\mathcal{V}=\bigoplus\limits_{j\in\mathbb{Z}}V_jq^j,$ then
\begin{equation}(\mathcal{V})_k=\bigoplus_{j\in\mathbb{Z}}V_jq^{\frac{j}{k}}.\label{lpok}\end{equation}

\textbf{The Functor $(\mbox{ })_k^{\Lambda}$}

Let $\Lambda^{var}_{(\underline{g}, \sigma)}(X)$ be the groupoid
with the same objects as $\Lambda^1_{(\underline{g}, \sigma)}(X)$
and morphisms
$$\coprod\limits_k\coprod_{(i_1, \cdots i_k), (j_1, \cdots
j_k)}\Lambda^k_G(g_{i_k}\cdots g_{i_1}, g_{j_k}\cdots
g_{j_1})\times X^{g_{i_k}\cdots g_{i_1}},$$ where $(i_1, \cdots
i_k)$ and $(j_1, \cdots j_k)$ go over all the $k-$cycles of
$\sigma$ respectively.

We can define a similar functor
\begin{equation}(\mbox{  })_k^{\Lambda}: \Lambda^{var}_{(\underline{g},
\sigma)}(X)\longrightarrow \Lambda^1_{(\underline{g},
\sigma)}(X)\label{rescalelambda}\end{equation} that is identity on
objects and sends each $[g, t]\in \Lambda^k_G(g_{i_k}\cdots
g_{i_1}, g_{j_k}\cdots g_{j_1})$ to $[g, \frac{t}{k}]\in
\Lambda^1_G(g_{i_k}\cdots g_{i_1}, g_{j_k}\cdots g_{j_1})$. We use
the same symbol $(\mbox{  })_k^{\Lambda}$ to denote the pull back
\begin{equation}K_{orb}(\Lambda^1_{(\underline{g},
\sigma)}(X))\longrightarrow K_{orb}(\Lambda^{var}_{(\underline{g},
\sigma)}(X)).\label{rescalekpower}\end{equation}


\bigskip

\textbf{The external product $\boxtimes$ }

Let $Y$ an $H$-space,  $(\underline{g}, \sigma)\in G\wr\Sigma_n$
and $(\underline{h}, \tau)\in G\wr\Sigma_m$.

Each $K^*_{orb}(d_{(\underline{g}, \sigma)}(X))$ is a
$\mathbb{Z}[q^{\pm}]-$algebra, as shown in Section \ref{setupad}.
The external product in the theory $K^*_{orb}(d_{(\underline{g},
\sigma)}(-))$ is defined to be the tensor product of
$\mathbb{Z}[q^{\pm}]-$algebras. The fibred product
$d_{(\underline{g},
\sigma)}(X)\times_{\mathbb{T}}d_{(\underline{h}, \tau)}(X)$ has
the same objects as $d_{(\underline{g}, \underline{h},
\sigma\tau)}(X)$ and is a subgroupoid of it.

So we have the K\"{u}nneth map
\begin{equation}K^*_{orb}(d_{(\underline{g},
\sigma)}(X))\otimes_{\mathbb{Z}[q^{\pm}]}K^*_{orb}(d_{(\underline{h},
\tau)}(X))\longrightarrow K^*_{orb}(d_{(\underline{g},
\sigma)}(X)\times_{\mathbb{T}}d_{(\underline{h}, \tau)}(X))\label{kunnethdc1}\end{equation} 
It is compatible with the K\"{u}nneth map (\ref{ku})  of the
quasi-elliptic cohomology in the sense that the diagram below
commutes.
\begin{equation}\begin{CD}[cols=0.1em]K^*_{orb}(d_{(\underline{g},
\sigma)}(X))\!\underset{\mathbb{Z}[q^{\pm}]}\otimes\!
K^*_{orb}\!(d_{(\underline{h}, \sigma)}(X)) @>>>
K^*_{orb}(d_{(\underline{g},
\sigma)}(X)\!\times_{\mathbb{T}}\!d_{(\!\underline{h},
\sigma\!)}(X))\\
@V{f_{(\underline{g},
\sigma)}^*\otimes_{\mathbb{Z}[q^{\pm}]}f_{(\underline{h},
\sigma)}^*}VV    @V{f_{(\underline{(g, h)}, \sigma)}^*}VV
\\ K^{*}_{\Lambda_{G\wr\Sigma_n}\!(\underline{g},
\sigma)}\!((X^n)^{(\underline{g}, \sigma)})
\!\!\!\underset{\mathbb{Z}[q^{\pm}]}\otimes\!\!\!
K^*_{\Lambda_{H\wr\Sigma_n}\!(\underline{h},
\sigma)}\!((Y^m)^{(\underline{h}, \sigma)}) @>>>
K^{*}_{\!\Lambda_{(\!G\!\times
\!H\!)\wr\Sigma_n}\!\!(\underline{(g, h)},
\sigma)}\!((X\!\times\!Y)^n)^{(\underline{(g, h)},
\sigma)}\end{CD}\label{kudqellc}\end{equation} where the
horizontal maps are K\"{u}nneth maps.

If we have a vector bundle
$E=\coprod\limits_k\coprod\limits_{(i_1, \cdots
i_k)}E_{g_{i_k}\cdots g_{i_1}}$ over $\Lambda^1_{(\underline{g},
\sigma)}(X)$,  the external product
$$\boxtimes_k\boxtimes_{(i_1, \cdots
i_k)}E_{g_{i_k}\cdots g_{i_1}}$$ is a 
vector bundler over $d_{(\underline{g}, \sigma)}(X)$. This defines
a map \begin{equation} K_{orb}(\Lambda^1_{(\underline{g},
\sigma)}(X))\longrightarrow K_{orb}(d_{(\underline{g},
\sigma)}(X))\label{lambdaproductd}\end{equation}

Composing all the functors as in (\ref{pgs1}) and (\ref{pgs2}), we get the explicit formula of $\mathbb{P}_{(\underline{g}, \sigma)}$ 
\begin{equation}\mathbb{P}_{(\underline{g},
\sigma)}(\mathcal{V})=f^*_{(\underline{g},
\sigma)}(\boxtimes_{k}\boxtimes_{(i_1,\cdots
i_k)}(\mathcal{V}_{g_{i_k}\cdots
g_{i_1}})_k).\label{ee}\end{equation}
$\mathbb{P}_{(\underline{g}, \sigma)}$ is natural. If
$(\underline{g}, \sigma)$ and $(\underline{h}, \tau)$ are
conjugate in $G\wr\Sigma_n$,  $\mathbb{P}_{(\underline{g},
\sigma)}(\mathcal{V})$ and $\mathbb{P}_{(\underline{h},
\tau)}(\mathcal{V})$ are isomorphic.

\begin{theorem}The family of maps $$\mathbb{P}_n=\prod_{(\underline{g},
\sigma)\in
(G\wr\Sigma_n)_{conj}}\mathbb{P}_{(\underline{g},\sigma)}:
QEll^*_G(X)\longrightarrow QEll^*_{G\wr\Sigma_n}(X^{\times n}),$$
satisfy

(i) $\mathbb{P}_1=$Id, $\mathbb{P}_0(x)=1$.

(ii) Let $x\in QEll^*_G(X)$, $(\underline{g}, \sigma)\in
G\wr\Sigma_n$ and $(\underline{h}, \tau)\in G\wr\Sigma_m$. The
external product of two power operations
$$\mathbb{P}_{(\underline{g}, \sigma)}(x)\boxtimes \mathbb{P}_{(\underline{h}, \tau)}(x)=res|^{\Lambda_{G\wr\Sigma_{m+n}}(\underline{g},\underline{h}; \sigma\tau)}
_{\Lambda_{G\wr\Sigma_n}(\underline{g},
\sigma)\times_{\mathbb{T}}\Lambda_{G\wr\Sigma_m}(\underline{h},
\tau)} \mathbb{P}_{(\underline{g},\underline{h};
\sigma\tau)}(x).$$

(iii) The composition of two power operations is
$$\mathbb{P}_{((\underline{\underline{h}, \tau}); \sigma)}(\mathbb{P}_m(x))=res|^{\Lambda_{G\wr\Sigma_{mn}}(\underline{\underline{h}}, (\underline{\tau}, \sigma))}
_{\Lambda_{(G\wr\Sigma_m)\wr\Sigma_n}((\underline{\underline{h},
\tau}); \sigma)}\mathbb{P}_{(\underline{\underline{h}},
(\underline{\tau}, \sigma))}(x)$$ where
$(\underline{\underline{h}, \tau})\in (G\wr \Sigma_m)^{\times n}$,
and $\sigma\in\Sigma_n$. $(\underline{\tau}, \sigma)$ is in
$\Sigma_m\wr\Sigma_n$, thus, can be viewed as an element in
$\Sigma_{mn}$.

(iv) $\mathbb{P}$ preserves external product. For $\underline{(g,
h)}=((g_1, h_1), \cdots (g_n, h_n))\in (G\times H)^{\times n}$,
$\sigma\in \Sigma_n$,
$$\mathbb{P}_{(\underline{(g, h)},\sigma)}(x\boxtimes y)=res|^{\Lambda_{G\wr\Sigma_n}(\underline{g}, \sigma)\times_{\mathbb{T}}\Lambda_{H\wr\Sigma_n}(\underline{h}, \sigma)}
_{\Lambda_{(G\times H)\wr\Sigma_n}(\underline{(g, h)},
\sigma)}\mathbb{P}_{(\underline{g}, \sigma)}(x)\boxtimes
\mathbb{P}_{(\underline{h}, \sigma)}(y).$$ \label{main1p}
\end{theorem}
\begin{proof}

We check each one respectively.

(i) When $n=1$, all the cycles of a permutation is 1-cycle.
$(\mbox{ })_1$ and the homeomorphism $f_{(g, (1))}$ are both
identity maps. Directly from the formula (\ref{ee}),
$\mathbb{P}_1(x)=x.$

\bigskip
(ii) \begin{align*}&\mathbb{P}_{(\underline{g},
\sigma)}(x)\boxtimes \mathbb{P}_{(\underline{h}, \tau)}(x)\\
=&f^*_{(\underline{g},
\sigma)}(\boxtimes_{k}\boxtimes_{(i_1,\cdots
i_k)}(x_{g_{i_k}\cdots g_{i_1}})_k)\boxtimes f^*_{(\underline{h},
\tau)}(\boxtimes_{j}\boxtimes_{(r_1,\cdots
r_j)}(x_{h_{r_j}\cdots h_{r_1}})_j)\\
=&res|^{\Lambda_{G\wr\Sigma_{m+n}}(\underline{g},\underline{h};
\sigma\tau)} _{\Lambda_{G\wr\Sigma_n}(\underline{g},
\sigma)\times_{\mathbb{T}}\Lambda_{G\wr\Sigma_m}(\underline{h},
\tau)} f^*_{(\underline{g}, \underline{h}; \sigma
\tau)}((\boxtimes_{k}\boxtimes_{(i_1,\cdots i_k)}(x_{g_{i_k}\cdots
g_{i_1}})_k)\\ &\boxtimes (\boxtimes_{j}\boxtimes_{(r_1,\cdots
r_j)}(x_{h_{r_j}\cdots h_{r_1}})_j)).\end{align*}where
$(i_1,\cdots i_k)$ goes over all the $k$-cycles of $\sigma$ and
$(r_1,\cdots r_j)$ goes over all the $j$-cycles of $\tau$ and
$(\mbox{                  })_k$ is the map cited in (\ref{lpok}).
The second step is from Proposition \ref{sim} (iii).

$$f^*_{(\underline{g}, \underline{h}; \sigma
\tau)}((\boxtimes_{k}\boxtimes_{(i_1,\cdots i_k)}(x_{g_{i_k}\cdots
g_{i_1}})_k)\boxtimes (\boxtimes_{j}\boxtimes_{(r_1,\cdots
r_j)}(x_{h_{r_j}\cdots h_{r_1}})_j))$$ is exactly
$$\mathbb{P}_{(\underline{g},\underline{h};
\sigma\tau)}(x).$$

\bigskip
(iii) Recall that for an element $(\underline{\tau},
\sigma)=(\tau_1, \cdots \tau_n, \sigma)\in\Sigma_{mn}$, it acts on
the set with $mn$ elements
$$\{(i, j)| 1\leq i\leq n, 1\leq j \leq m\}$$ in this way: $$(\underline{\tau}, \sigma)\cdot(i, j)=(\sigma(i),
\tau_{\sigma(i)}(j)).$$ That also shows how to view it as an
element in $\Sigma_{mn}$.

Then for any integer $q$, \begin{equation}(\underline{\tau},
\sigma)^q\cdot (i, j)=(\sigma^q(i),
\tau_{\sigma^q(i)}\tau_{\sigma^{q-1}(i)}\cdots
\tau_{\sigma(i)}(j)).\label{not}\end{equation}
\smallskip

To find all the cycles of $(\underline{\tau}, \sigma)$ is exactly
to find all the orbits of the action by $(\underline{\tau},
\sigma)$. If $i$ belongs to an $s$-cycle of $\sigma$ and $j$
belongs to a $r$-cycle of
$\tau_{\sigma^s(i)}\tau_{\sigma^{s-1}(i)}\cdots \tau_{\sigma(i)}$,
then the orbit containing $(i, j)$ has $sr$ elements by
(\ref{not}). In other words, $(i_1, \cdots i_s)$ is an $s$-cycle
of $\sigma$ and $(j_1, \cdots j_r)$ is a $r$-cycle of
$\tau:=\tau_{i_s}\cdots \tau_{i_1}$ if and only if
\begin{align*}\bigg(&(i_1, \tau_{i_1}(j_{r-1}))(i_2, \tau_{i_2}\tau_{i_1}(j_{r-1}))\cdots (i_s, j_r)\\
&(i_1, \tau_{i_1}(j_{r-2}))(i_2, \tau_{i_2}\tau_{i_1}(j_{r-2}))\cdots (i_s, j_{r-1})\\
&\cdots\\
&(i_1, \tau_{i_1}(j_{1}))(i_2,
\tau_{i_2}\tau_{i_1}(j_{1}))\cdots (i_s, j_2)\\
&(i_1, \tau_{i_1}(j_{r}))(i_2, \tau_{i_2}\tau_{i_1}(j_{r}))\cdots
(i_s, j_1)\bigg)\\\end{align*} is an $sr$-cycle of
$(\underline{\tau}, \sigma)$.

\begin{align*}&\mathbb{P}_{((\underline{\underline{h}, \tau});
\sigma)}(\mathbb{P}_m(x))\\ =&f^*_{((\underline{\underline{h},
\tau}); \sigma)}[\boxtimes_k\boxtimes_{(i_1,\cdots
i_k)}(\mathbb{P}_{((\underline{h}_{i_k},\tau_{i_k})\cdots(\underline{h}_{i_1},\tau_{i_1}))}(x))_k]\\
=&f^*_{((\underline{\underline{h}, \tau});
\sigma)}[\boxtimes_k\boxtimes_{(i_1, \cdots
i_k)}[f^*_{((\underline{h}_{i_k},\tau_{i_k})\cdots(\underline{h}_{i_1},\tau_{i_1}))}(\boxtimes_r\boxtimes_{(j_1,
\cdots
j_r)}(x_{H_{\underline{i}\underline{j}}})_r)]_k]\\
=&(f^*_{((\underline{\underline{h}, \tau}); \sigma)}\circ
\prod_{k, (i_1, \cdots
i_k)}f^*_{((\underline{h}_{i_k},\tau_{i_k})\cdots(\underline{h}_{i_1},\tau_{i_1}))})[\boxtimes_{k,
(i_1,\cdots i_k)}\boxtimes_{r, (j_1, \cdots
j_r)}(x_{H_{\underline{i},
\underline{j}}})_{kr}]\\
=&f^*_{(\underline{\underline{h}}, (\underline{\tau},
\sigma))}[\boxtimes_{k, (i_1,\cdots i_k)}\boxtimes_{r, (j_1,
\cdots j_r)}(x_{H_{\underline{i},
\underline{j}}})_{kr}]\end{align*} where
\begin{align*}H_{\underline{i}\underline{j}}:=\mbox{  }&h_{i_k, j_1}h_{i_{k-1}, \tau^{-1}_{i_k}(j_1)}\cdots h_{i_1, (\tau_{i_k}\cdots\tau_{i_2})^{-1}(j_1)}\\
&h_{i_k, j_2}h_{i_{k-1}, \tau^{-1}_{i_k}(j_2)}\cdots h_{i_1, (\tau_{i_k}\cdots\tau_{i_2})^{-1}(j_2)}\\
&\cdots\\
&h_{i_k, j_r}h_{i_{k-1}, \tau^{-1}_{i_k}(j_r)}\cdots h_{i_1,
(\tau_{i_k}\cdots\tau_{i_2})^{-1}(j_r)}\\=\mbox{  }&h_{i_k,
j_1}h_{i_{k-1},
\tau_{i_{k-1}}\cdots\tau_{i_2}\tau_{i_1}(j_1)}\cdots h_{i_1,
\tau_{i_1}(j_r)}\\
&h_{i_k, j_2}h_{i_{k-1},
\tau_{i_{k-1}}\cdots\tau_{i_2}\tau_{i_1}(j_2)}\cdots h_{i_1,
\tau_{i_1}(j_1)}\\
&\cdots\\
&h_{i_k, j_r}h_{i_{k-1},
\tau_{i_{k-1}}\cdots\tau_{i_2}\tau_{i_1}(j_{r-1})}\cdots h_{i_1,
\tau_{i_1}(j_{r-1})}
\end{align*}
where $(i_1, \cdots i_k)$ goes over all the $k$-cycles of
$\sigma\in\Sigma_m$ and $(j_1, \cdots j_r)$ goes over all the
$r$-cycles of $\tau_{i_k}\cdots\tau_{i_1}\in\Sigma_n$. The last step 
is by Proposition 4.11 in \cite{Gan07}.

$$f^*_{(\underline{\underline{h}}, (\underline{\tau},
\sigma))}[\boxtimes_{k, (i_1,\cdots i_k)}\boxtimes_{r, (j_1,
\cdots j_r)}(x_{H_{\underline{i}, \underline{j}}})_{kr}]$$ is the
same space as $\mathbb{P}_{(\underline{\underline{h}},
(\underline{\tau}, \sigma))}(x)$, 
but the action is restricted by
$$res|^{\Lambda_{G\wr\Sigma_{mn}}(\underline{\underline{h}}, (\underline{\tau}, \sigma))}
_{\Lambda_{(G\wr\Sigma_m)\wr\Sigma_n}((\underline{\underline{h},
\tau}); \sigma)}.$$

\bigskip
(iv)We have
\begin{align*}&\mathbb{P}_{(\underline{(g,h)},\sigma)}(x\boxtimes
y)=f^*_{(\underline{(g, h)},
\sigma)}(\boxtimes_{k}\boxtimes_{(i_1,\cdots
i_k)}((x\boxtimes y)_{(g_{i_k}\cdots g_{i_1}, h_{i_k}\cdots h_{i_1})})_k)\\ &\\
=&f^*_{(\underline{(g, h)},
\sigma)}(\boxtimes_{k}\boxtimes_{(i_1,\cdots
i_k)}(x_{g_{i_k}\cdots g_{i_1}})_k\boxtimes (y_{h_{i_k}\cdots
h_{i_1}})_k)\\ &\\ =&f^*_{(\underline{(g, h)},
\sigma)}(\boxtimes_{k}\boxtimes_{(i_1,\cdots
i_k)}(x_{g_{i_k}\cdots
g_{i_1}})_k)\boxtimes(\boxtimes_{j}\boxtimes_{(r_1,\cdots r_j)}
(y_{h_{r_j}\cdots h_{r_1}})_j)\\ &\\
=&res|^{\Lambda_{G\wr\Sigma_n}(\underline{g},
\sigma)\times_{\mathbb{T}}\Lambda_{H\wr\Sigma_n}(\underline{h},
\sigma)} _{\Lambda_{(G\times H)\wr\Sigma_n}(\underline{(g, h)},
\sigma)}(f^*_{(\underline{g}, \sigma)}\times f^*_{(\underline{h},
\sigma)})(\boxtimes_{k}\boxtimes_{(i_1,\cdots
i_k)}(x_{g_{i_k}\cdots g_{i_1}})_k)\\
&\boxtimes(\boxtimes_{j}\boxtimes_{(r_1,\cdots r_j)}
(y_{h_{r_j}\cdots h_{r_1}})_j)\\ &\\
=&res|^{\Lambda_{G\wr\Sigma_n}(\underline{g},
\sigma)\times_{\mathbb{T}}\Lambda_{H\wr\Sigma_n}(\underline{h},
\sigma)} _{\Lambda_{(G\times H)\wr\Sigma_n}(\underline{(g, h)},
\sigma)}f^*_{(\underline{g},\sigma)}[\boxtimes_{k}\boxtimes_{(i_1,\cdots
i_k)}(x_{g_{i_k}\cdots g_{i_1}})_k]\\ &\boxtimes
f^*_{(\underline{h},\sigma)}[\boxtimes_{j}\boxtimes_{(r_1,\cdots
r_j)} (y_{h_{r_j}\cdots h_{r_1}})_j],\end{align*}where $(i_1,
\cdots i_k)$ goes over all the $k-$cycles of $\sigma$ and $(r_1,
\cdots r_j)$ goes over all the $j-$cycles of $\sigma$.  It equals
to
$$res|^{\Lambda_{G\wr\Sigma_n}(\underline{g}, \sigma)\times_{\mathbb{T}}\Lambda_{H\wr\Sigma_n}(\underline{h}, \sigma)}
_{\Lambda_{(G\times H)\wr\Sigma_n}(\underline{g},
\sigma)}\mathbb{P}_{(\underline{g}, \sigma)}(x)\boxtimes
\mathbb{P}_{(\underline{h}, \sigma)}(y).$$

\end{proof}

\begin{example}Let $G$ be the trivial group and $X$ a space. Let $\sigma\in\Sigma_n$.  Then $QEll^*_G(X)=
K^*_{\mathbb{T}}(X)$. The functor $f_{(\underline{1}, \sigma)}$
gives the homeomorphism
$$(X^{\times n})^{(\underline{1}, \sigma)}\cong
\prod_k\prod_{(i_1, \cdots i_k)}X,$$ where the second direct
product goes over all the $k$-cycles of $\sigma$. By (\ref{ee}),
the power operation is
$$\mathbb{P}_{(\underline{1}, \sigma)}(x)=\boxtimes_k\boxtimes_{(i_1, \cdots
i_k)}(x)_k.$$

When $n=2$, $\mathbb{P}_{(\underline{1}, (1)(1))}(x)=x\boxtimes x$
and $\mathbb{P}_{(\underline{1}, (12))}(x)=(x)_2$.

When $n=3$, $\mathbb{P}_{(\underline{1}, (1)(1)(1))}(x)=x\boxtimes
x\boxtimes x$, $\mathbb{P}_{(\underline{1},
(12)(1))}(x)=(x)_2\boxtimes x$, and $\mathbb{P}_{(\underline{1},
(123))}(x)=(x)_3$.

When $n=4$, $\mathbb{P}_{(\underline{1}, (1)(1)(1)(1))}(x)=
x\boxtimes x\boxtimes x\boxtimes x$, $\mathbb{P}_{(\underline{1},
(12))}(x)=(x)_2\boxtimes x\boxtimes x$,
$\mathbb{P}_{(\underline{1}, (123))}(x)=(x)_3\boxtimes x$,
$\mathbb{P}_{(\underline{1}, (1234))}(x)=(x)_4$, and
$\mathbb{P}_{(\underline{1}, (12)(34))}(x)=(x)_2\boxtimes (x)_2.$
Note that there is a $\Sigma_2-$action permuting the two $(x)_2$
in $\mathbb{P}_{(\underline{1}, (12)(34))}(x)$.

\label{pointsymmpower}

\end{example}

\begin{remark}
We have the relation between equivariant Tate K-theory and
quasi-elliptic cohomology
\begin{equation}QEll_G(X)\otimes_{\mathbb{Z}[q^{\pm}]}\mathbb{Z}((q))\cong (K_{Tate})_G(X).\end{equation}

It extends uniquely to  a power operation for Tate K-theory
$$
QEll_G(X)\otimes_{\mathbb{Z}[q^{\pm}]}\mathbb{Z}((q))\longrightarrow
QEll_{G\wr\Sigma_n}(X^{\times
n})\otimes_{\mathbb{Z}[q^{\pm}]}\mathbb{Z}((q))
$$
which  is the stringy power operation $P^{string}_n$ constructed
in Definition 5.10, \cite{Gan07}. It is elliptic in the sense of
\cite{AHS04}. \label{stringyrelation}

\end{remark}

\section{Orbifold quasi-elliptic cohomology and its power operation}\label{orbifoldquasibeforepower}

The elliptic cohomology of orbifolds involves a rich interaction
between the orbifold structure and the elliptic curve. Ganter
explores this interaction in the case of the Tate curve in
\cite{Gan13}, describing $K_{Tate}$ for an orbifold $X$ in terms
of the equivariant K-theory and the groupoid structure of $X$.

In Section \ref{sb1} we give a description of orbifold
quasi-elliptic cohomology. In Section \ref{sb2} we   discuss the
inertia groupoid of symmetric power and the groupoids needed for
the construction of the power operation in Section \ref{s2}.

\subsection{Definition}\label{sb1}

We have two ways to define orbifold quasi-elliptic cohomology. The
first one is motivated by Ganter's definition of orbifold Tate
K-theory in Section 2, \cite{Gan13}. The other one is a
generalization of the definition of quasi-elliptic cohomology in
Section \ref{orbqec}.

We consider the category of groupoids $\mathcal{G}pd$ as a
2-category  with small topological groupoids as the objects and
with
$$\mbox{1Hom}(X, Y)=Fun(X, Y).$$  This 2-category is different from that in Section 3
\cite{LerStack}. Let $\mathcal{G}pd^{cen}$ denote the 2-category
of centers of groupoids defined in Section 2, \cite{Gan13}. Ganter
constructed in Example 2.3 \cite{Gan13} a 2-functor for any $k\in
\mathbb{Z}$
\begin{align*}\mathcal{G}pd &\longrightarrow \mathcal{G}pd^{cen}\\
X&\mapsto (I(X), \xi^k)
\end{align*} where $\xi^k$ is the center element of the inertia groupoid $I(X)$ sending
$(x, g)$ to $(x, g^k)$. We use $\xi$ to denote $\xi^1$.

Let $\mbox{pt}/\!\!/\mathbb{R}\times_{1\sim\xi}I(X)$ denote the
groupoid
$$(\mbox{pt}/\!\!/\mathbb{R})\times I(X)/\sim$$ with
$\sim$ generated by $1\sim\xi$.

\begin{definition}For any topological groupoid $X$, the quasi-elliptic
cohomology $QEll^*(X)$ is the orbifold K-theory
\begin{equation}K^*_{orb}(\mbox{pt}/\!\!/\mathbb{R}\times_{1\sim
\xi}I(X)).\label{orbqecnora}\end{equation}

In other words, 
for a topological groupoid $X$, $QEll(X)$ is defined to be a
subring of  $K_{orb}(X)\llbracket q^{\pm\frac{1}{|\xi|}}
\rrbracket$ that is the Grothendieck group of finite sums
$$\sum_{a\in\mathbb{Q}} V_a q^a$$ satisfying: 
$$\mbox{for each }a\in\mathbb{Q} \mbox{, the coefficient }V_a \mbox{ is an }e^{2\pi ia}-\mbox{eigenbundle of }\xi.$$

In the global quotient case,
$$QEll^*(X/\!\!/G)=QEll^*_G(X).$$ \label{orbqecnoradef}\end{definition}

\bigskip

In addition, for any topological groupoid $X$, we can also
consider the category
$$Loop_1(X):=Bibun(S^1/\!\!/\ast, X)$$ and formulate
$Loop_1^{ext}(X)$ by adding the rotation action by circle, as the
construction in Section \ref{orbifoldloop}.  Afterwards we can
construct the subgroupoid $\Lambda(X)$ of  $Loop_1^{ext}(X)$
consisting of the constant loops, which is isomorphic to
$\mbox{pt}/\!\!/\mathbb{R}\times_{1\sim \xi}I(X)$. So in this way
we give an equivalent definition of orbifold quasi-elliptic
cohomology.

\subsection{Symmetric powers of orbifolds and its inertia
groupoid}\label{sb2}

In this section  we  introduce the groupoids necessary for the
construction of the power operation. In  Lemma \ref{eq1},
 \ref{eq2} and \ref{eq3} we show the relation between
them.

For groupoids  like $\mbox{pt}/\!\!/\mathbb{R}\times_{k\sim\xi}X$,
instead of the total symmetric power (Definition 3.1,
\cite{Gan13}) $S(\mbox{pt}/\!\!/\mathbb{R}\times_{k\sim\xi}X)$, we
consider a subgroupoid
$$S^R(\mbox{pt}/\!\!/\mathbb{R}\times_{k\sim\xi}X)$$ of
it.

\begin{definition}[The groupoid $S^R(\mbox{pt}/\!\!/\mathbb{R}\times_{k\sim\xi}X)$]
Let $$\rho_k:
\mbox{pt}/\!\!/\mathbb{R}\times_{k\sim\xi}X\longrightarrow
\mbox{pt}/\!\!/(\mathbb{R}/\mathbb{Z})$$ be the functor sending
all the objects to the single point, and an arrow $$[g, t]$$ to
$$t\mbox{ mod }\mathbb{Z}.$$

Let
$\times_{\mathbb{R}}(\mbox{pt}/\!\!/\mathbb{R}\times_{k\sim\xi}X)$
denote the limit of the diagram of groupoids
$$\xymatrix{\mbox{pt}/\!\!/\mathbb{R}\times_{k\sim\xi}X \ar[r]^{\rho_k}
&\mbox{pt}/\!\!/(\mathbb{R}/\mathbb{Z})
&\mbox{pt}/\!\!/\mathbb{R}\times_{k\sim\xi}X \ar[l]_{\rho_k}}.$$

Let
$$\times^n_{\mathbb{R}}(\mbox{pt}/\!\!/\mathbb{R}\times_{k\sim\xi}X)$$
denote the limit of $n$ morphisms $\rho_k$s. It inherits a
$\Sigma_n-$action on it by permutation from that on the product
$(\mbox{pt}/\!\!/\mathbb{R}\times_{k\sim\xi}X)^{\times n}$.

Let $S_n^R(\mbox{pt}/\!\!/\mathbb{R}\times_{k\sim\xi}X)$ denote
the groupoid with the same objects as
$$\times^n_{\mathbb{R}}(\mbox{pt}/\!\!/\mathbb{R}\times_{k\sim\xi}X)$$
and morphisms of the form $([g_1, t_1], \cdots [g_n, t_n];
\sigma)$ with $([g_1, t_1], \cdots [g_n, t_n])$ a morphism in
$\times^n_{\mathbb{R}}(\mbox{pt}/\!\!/\mathbb{R}\times_{k\sim\xi}X)$
and $\sigma\in\Sigma_n$. This new groupoid
$S_n^R(\mbox{pt}/\!\!/\mathbb{R}\times_{k\sim\xi}X)$ is a
subgroupoid of $$(\mbox{pt}/\!\!/\mathbb{R}\times_{k\sim\xi}X)\wr
\Sigma_n.$$

Define
\begin{equation}S^R(\mbox{pt}/\!\!/\mathbb{R}\times_{k\sim\xi}X):=\coprod\limits_{n\geq
0}S_n^R(\mbox{pt}/\!\!/\mathbb{R}\times_{k\sim\xi}X).\end{equation}
\end{definition}

The triple $$(S^R(\mbox{pt}/\!\!/\mathbb{R}\times_{k\sim\xi}X),
\ast, (\mbox{  }))$$ is  a symmetric monoid  where $\ast$ is the
concatenation and the unit $(\mbox{   })$ is the unique object in
$X\wr\Sigma_0$. $S^R(\mbox{pt}/\!\!/\mathbb{R}\times_{k\sim\xi}X)$
is the symmetric product that we will use to formulate the power
operation.

\begin{lemma}Let $\Phi_k(X)$ denote the groupoid in Definition 3.3, \cite{Gan13}, and $\phi_k\in$ \\ Center$(\Phi_k)$ denote the
restriction of $S_k(\xi)$ to $\Phi_k$.   For each integer $k\geq
1$, there is an equivalence between
$$\mbox{pt}/\!\!/\mathbb{R}\times_{1\sim\phi_k}\Phi_k(X)$$ and the
groupoid $\mbox{pt}/\!\!/\mathbb{R}\times_{1\sim
\xi^{\frac{1}{k}}}I(X)[\xi^{\frac{1}{k}}]$ which identifies
$\phi_k$ with $\xi^{\frac{1}{k}}$. Here $\xi^{\frac{1}{k}}$ is an
added element such that the composition of $k$
$\xi^{\frac{1}{k}}$s is $\xi$. \label{eq1}\end{lemma}


\begin{proof}

We can define a functor $$A_k:
\mbox{pt}/\!\!/\mathbb{R}\times_{1\sim\phi_k}\Phi_k(X)\longrightarrow
\mbox{pt}/\!\!/\mathbb{R}\times_{1\sim
\xi^{\frac{1}{k}}}I(X)[\xi^{\frac{1}{k}}]$$ by sending an object
$(\underline{x}, \underline{g}, (1 2 \cdots k))$ to $(x_1,
g_k\cdots g_1)$ and sending a morphism $[\underline{h}, (1 2\cdots
k)^m, t]$ to $$[h_kg^{-1}_{1-m}\cdots g^{-1}_{k-1}g^{-1}_k,
m+t].$$

Recall $h_kg^{-1}_{1-m}\cdots g^{-1}_{k-1}g^{-1}_k$ conjugates
$g_k\cdots g_1$ to itself. It is the element
$$\beta^{\underline{h}, \mbox{Id}}_{(1 2\cdots k), (1 2\cdots
k)}$$ defined in (\ref{betadef}). The functor $A_k$ is an
isomorphism, as implied in the proof of Theorem \ref{iptm}.

\end{proof}

Let $\Phi(X):=\coprod\limits_{k\geq 1}\Phi_k(X)$. Let
$\phi:=\coprod\limits_{k\geq 1}\phi_k\in\mbox{Center}(\Phi)$
denote the restriction of $S(\xi)$ to $\Phi$.

Theorem \ref{iptm} can be reinterpreted as Lemma \ref{eq2}.
\begin{lemma}
The groupoid
$S^R(\coprod\limits_k\mbox{pt}/\!\!/\mathbb{R}\times_{1\sim
\xi^{\frac{1}{k}}}I(X)[\xi^{\frac{1}{k}}])$ is equivalent to
$$\mbox{pt}/\!\!/\mathbb{R}\times_{1\sim S(\xi)} I(S(X)).$$
\label{eq2}
\end{lemma}
The proof is similar to that of Theorem \ref{iptm}.

\begin{lemma}We have an equivalence of 
 groupoids $$Q^R: S^R(\mbox{pt}/\!\!/\mathbb{R}\times_{1\sim \phi}\Phi(X))\longrightarrow \mbox{pt}/\!\!/\mathbb{R}\times_{1\sim S(\xi)}
I(S(X)),$$ which is natural in $X$ and satisfies
$$Q^RS^R(\phi)=S(\xi) Q^R.$$\label{eq3}\end{lemma} 
\begin{proof}Let $I$ be the inclusion $$\mbox{pt}/\!\!/\mathbb{R}\times_{1\sim \phi}\Phi(X)\longrightarrow \mbox{pt}/\!\!/\mathbb{R}\times_{1\sim S(\xi)}
I(S(X)).$$ Let $\epsilon$ be the counit of the adjunction
$(S,\ast, (\mbox{  }))\dashv \mbox{forget}$. Let $Q$ denote the
composition
$$S(\mbox{pt}/\!\!/\mathbb{R}\times_{1\sim \phi}\Phi(X))\buildrel{S(I)}\over\longrightarrow S(\mbox{pt}/\!\!/\mathbb{R}\times_{1\sim S(\xi)}
I(S(X)))\buildrel{\epsilon}\over\longrightarrow\mbox{pt}/\!\!/\mathbb{R}\times_{1\sim
S(\xi)} I(S(X)).$$ Let $Q^R$ be the restriction of $Q$ to the
subgroupoid $S^R(\mbox{pt}/\!\!/\mathbb{R}\times_{1\sim
\phi}\Phi(X))$, i.e. the composition \begin{align*}Q^R:
S^R(\mbox{pt}/\!\!/\mathbb{R}\times_{1\sim
\phi}\Phi(X))&\buildrel{S^R(I)}\over\longrightarrow
S^R(\mbox{pt}/\!\!/\mathbb{R}\times_{1\sim S(\xi)} I(S(X)))\\
&\buildrel{\mbox{restriction of
}\epsilon}\over\longrightarrow\mbox{pt}/\!\!/\mathbb{R}\times_{1\sim
S(\xi)} I(S(X)).\end{align*}

The essential image of $I$ consists exactly of the indecomposable
objects of $\mbox{pt}/\!\!/\mathbb{R}\times_{1\sim S(\xi)}
I(S(X))$, thus, both $Q$ and $Q^R$ are essentially surjective.


$Q$ is not fully faithful but $Q^R$ is. This is why we need the
product $S^R$ instead of $S$.

\end{proof}


\subsection{Power Operation for orbifold quasi-elliptic
cohomology}\label{s2}

In this section we construct the total power operation for the
orbifold quasi-elliptic cohomology
$$P^{Ell}: QEll(X)\longrightarrow QEll(SX)$$ in (\ref{main2porb}), which satisfy the axioms that Ganter
formulated in Definition 3.9, \cite{Gan13} for power operations
for orbifold theories. The power operation we constructed in
Section \ref{s2complete} is a special case of it for $G-$spaces.

\begin{example}\label{atiyahell} We can construct Atiyah's power
operation for orbifold quasi-elliptic cohomology.

Let $V$ be an orbifold vector bundle over the orbifold
$$\mbox{pt}/\!\!/\mathbb{R}\times_{1\sim \xi}I(X),$$ thus, $V$
represents an element in $QEll(X)$. Then $$P_n(V):=
V^{\otimes_{\mathbb{Z}[q^{\pm}]} n}$$ is an orbifold vector bundle
over
$$S^R(\mbox{pt}/\!\!/\mathbb{R}\times_{1\sim \xi}I(X))\cong \mbox{pt}/\!\!/\mathbb{R}\times_{1\sim \xi}SI(X).$$ So $P_n(V)$
is in $QEll^*(S(X))$.

$P=(P_n)_{n\geq 0}$ satisfies the axioms of a total power
operation. \label{apel}\end{example}


Before the construction of the power operation of $QEll$, we
introduce several maps necessary for the construction of the power
operation.

Let $X$ be an orbifold groupoid 
and $k\geq 1$ an integer. We define
the map
\begin{align}
s_k:
K_{orb}(\mbox{pt}/\!\!/\mathbb{R}\times_{1\sim\xi}I(X))&\longrightarrow
K_{orb}(\mbox{pt}/\!\!/\mathbb{R}\times_{k\sim\xi}I(X))\\
[\sum V_aq^a]&\mapsto[\sum V_a q^{\frac{a}{k}}]\end{align} and
\begin{equation}\coprod\limits_k s_k:
K_{orb}(\mbox{pt}/\!\!/\mathbb{R}\times_{1\sim\xi}I(X))\longrightarrow
K_{orb}(\coprod\limits_k(\mbox{pt}/\!\!/\mathbb{R}\times_{k\sim\xi}I(X))).\end{equation}

The functor $$(\mbox{  })_k: \Lambda_{(\underline{g},
\sigma)}(X)\longrightarrow \Lambda^1_{(\underline{g},
\sigma)}(X)$$ defined in (\ref{lpok})  
is a special local case of $s_k$ when $X$ is a $G-$space and
$(\underline{g}, \sigma)$ is fixed.

\bigskip

Let $\theta: QEll(X)\longrightarrow
K_{orb}(\mbox{pt}/\!\!/\mathbb{R}\times_{1\sim\phi}\Phi(X))$    
 be the additive operation whose $k-$th component is
$A_k^*\circ s_k$, where $A_k$ is the equivalence defined in Lemma
\ref{eq1}.

\bigskip

Now we are ready to define the total power operation $P^{Ell}$ of
$QEll^*$ as the composition below:

\begin{equation}\xymatrix{ QEll(X)\ar[r]^>>>>>{\theta}
&K_{orb}(\mbox{pt}/\!\!/\mathbb{R}\times_{1\sim\phi}\Phi(X))\ar[r]^>>>>>{P}
&K_{orb}(S^R(\mbox{pt}/\!\!/\mathbb{R}\times_{1\sim
\phi}\Phi(X)))\ar[d]^>>>>>>{(Q^{R *})^{-1}} \\ &&QEll(SX).}
\label{main2porb}\end{equation}

\begin{theorem}$P^{Ell}$ satisfies the axioms of a total power
operation in Definition 3.9 \cite{Gan13}.\end{theorem}

\begin{proof}
From the definition of $P^{Ell}$, we can see it is a well-defined
natural transformation $QEll\Rightarrow QEll\circ S$ and is a
comodule over the comonad $(-)\circ S$.

In addition, the functor $\theta$ has the property of additivity
\begin{align*}\theta: QEll(X\sqcup Y)&\longrightarrow QEll(\Phi(X)\sqcup\Phi(Y))\\ (a, b)&\mapsto (\theta(a), \theta(b)). \end{align*}
The power operation $P$ defined in Example \ref{atiyahell} has the
exponential property. Therefore, $P^{Ell}$ has the exponential
property. So $P^{Ell}$ is a total power operation.
\end{proof}

\begin{remark}
Let $X/\!\!/G$ be a quotient orbifold. The power operation we
construct in Section \ref{111} for quotient orbifolds is in fact
the one below.

$\mathbb{P}: QEll^*(X/\!\!/G)\buildrel{\coprod\limits_k
s_k}\over\longrightarrow K^*_{orb}(\coprod\limits_k
\mbox{pt}/\!\!/\mathbb{R}\times_{1\sim
\xi^{\frac{1}{k}}}I(X/\!\!/G)[\xi^{\frac{1}{k}}])\buildrel{P}\over\longrightarrow
\\ K^*_{orb}(S^R(\coprod\limits_k
\mbox{pt}/\!\!/\mathbb{R}\times_{1\sim
\xi^{\frac{1}{k}}}I(X/\!\!/G)[\xi^{\frac{1}{k}}]))\buildrel{J^*}\over\longrightarrow
QEll^*(S(X/\!\!/G))$ where $J$ is constructed from the functors
$J_{(\underline{g}, \sigma)}$ in the proof of Theorem \ref{iptm}.

For global quotient orbifolds, $P^{Ell}$ and $\mathbb{P}$ are the
same up to isomorphism. The diagram

$$\xymatrix{QEll^*(X/\!\!/G)\ar[d]^{\theta} &QEll^*(S(X/\!\!/G)) \\ K_{orb}(\mbox{pt}/\!\!/\mathbb{R}\underset{1\sim\phi}\times
I(\Phi(X/\!\!/G)))\ar[r]^{P}
&K_{orb}(S^R(\mbox{pt}/\!\!/\mathbb{R}\underset{1\sim
\phi}\times\Phi(X/\!\!/G)))\ar[u]^>>>>>{(Q^{R*})^{-1}}
\\
 K_{orb}(\coprod\limits_k \mbox{pt}/\!\!/\mathbb{R}\!\!\underset{1\sim
\xi^{\frac{1}{k}}}\times\!\!
I(X/\!\!/G)[\xi^{\frac{1}{k}}])\ar[r]^>>>>>{P}\ar[u]^{\coprod\limits_k
A_k^*} &K_{orb}(S^R(\coprod\limits_k
\mbox{pt}/\!\!/\mathbb{R}\!\!\underset{1\sim
\xi^{\frac{1}{k}}}\times\!\!
I(X/\!\!/G)[\xi^{\frac{1}{k}}]))\ar[u]^{S^R(\coprod\limits_k
A_k^*)} }$$ commutes. The vertical maps $\coprod\limits_k A_k^*$
and $S^R(\coprod\limits_k A_k^*)$ are both equivalences of
groupoids. The horizontal maps are the power operation defined in
Example \ref{atiyahell}.

\end{remark}

\section{Finite subgroups of the Tate curve}
\label{stricklandtorsion}

Strickland showed in \cite{Str98} that the quotient of the Morava
E-theory of the symmetric group by a certain transfer ideal can be
identified with the product of rings $\prod\limits_{k\geq 0}R_k$
where each $R_k$ classifies subgroup-schemes of degree $p^k$  in
the  formal group associated
to $E^0\mathbb{C}P^{\infty}$.  
In this section we prove similar conclusions for Tate K-theory and
quasi-elliptic cohomology. The main conclusion for Section
\ref{stricklandtorsion} is Theorem \ref{stricktheorem}.

\subsection{Background}\label{tatecurve}

In this section we introduce the Tate curve and its finite
subgroups. The main references are Section 2.6, \cite{AHS} and
Section 8.7, 8.8, \cite{KM85}.

An elliptic curve over the complex numbers $\mathbb{C}$ is a
connected Riemann surface, i.e. a connected compact 1-dimensional
complex manifold, of genus 1. By the uniformization theorem every
elliptic curve over $\mathbb{C}$ is analytically isomorphic to a
1-dimensional complex torus, and can be expressed as
$$\mathbb{C}^*/q^{\mathbb{Z}}$$ with $q\in\mathbb{C}$ and
$0<|q|<1$, where $\mathbb{C}^*$ is the multiplicative group
$\mathbb{C}\backslash\{0\}.$

The Tate curve  $Tate(q)$  is the elliptic curve $$E_q:
y^2+xy=x^3+a_4x+a_6$$ whose coefficients are given by the formal
power series in $\mathbb{Z}((q))$
$$a_4=-5\sum_{n\geqslant 1} n^3q^n/(1-q^n)   \,\,\,\,\,\,\,\,\,\,\,\,\,\,\,   a_6=-\frac{1}{12}\sum_{n\geqslant 1}(7n^5+5n^3)q^n/(1-q^n).$$



Before we talk about the torsion part of $Tate(q)$, we recall a
smooth one-dimensional commutative group scheme $T$ over
$\mathbb{Z}[q^{\pm}]$. It sits in a short exact sequence of
group-schemes over $\mathbb{Z}[q^{\pm}]$
$$0\longrightarrow \mathbb{G}_m\longrightarrow T\longrightarrow \mathbb{Q}/\mathbb{Z}\longrightarrow 0.$$

 The $N-$torsion points
$T[N]$ of it is the disjoint union of $N$ schemes $T_0[N]$,
$\cdots$ $T_{N-1}[N]$, where
$$T_i[N]=\mbox{Spec}(\mathbb{Z}[q^{\pm}][x]/(x^N-q^i)).$$ It fits
into a short exact sequence $$0\longrightarrow
\mu_N\buildrel{a_N}\over\longrightarrow
T[N]\buildrel{b_N}\over\longrightarrow
\mathbb{Z}/N\mathbb{Z}\longrightarrow 0,$$ 
The canonical extension structure on $T(N)$ is compatible with an
alternating paring of $\mathbb{Z}[q^{\pm}]-$group schemes $e_N:
T(N)\times T(N)\longrightarrow \mu_N$ in the sense that
$$e_N(a_N(x), y)= x^{b_N(y)}\mbox{,   for any  }\mathbb{Z}[q^{\pm}]-\mbox{algebra } R\mbox{   and  any }x\in\mu_N(R).$$

We have the conclusion below, which is Theorem 8.7.5, \cite{KM85}.

\begin{theorem}
There exists a faithfully flat $\mathbb{Z}[q^{\pm}]-$algebra $R$,
an elliptic curve $E/R$, and an isomorphism of ind-group-schemes
over $R$ $$T_{torsion}\otimes_{\mathbb{Z}[q^{\pm}]}
R\buildrel\sim\over\longrightarrow E_{tors},$$ such that for every
$N\geq 1$, the isomorphism on $N-$division points $T[N]\otimes
R\buildrel\sim\over\longrightarrow E[N]$ is compatible with
$e_N-$pairings. \label{torsionT}
\end{theorem}

Thus, we have the unique isomorphism of ind-group-schemes on
$\mathbb{Z}((q))$ $$T_{torsion}\otimes_{\mathbb{Z}[q^{\pm}]}
\mathbb{Z}((q))\buildrel\sim\over\longrightarrow Tate(q)_{tors}.$$
The isomorphism is compatible with the canonical extension
structure: for each $N\geq 1$, $$\xymatrix{0 \ar[r]
&\mu_N\ar[d]_{=}\ar[r]
&T[N] \ar[r]\ar[d]_{\cong} &\mathbb{Z}/N\mathbb{Z} \ar[r] \ar[d]_= &0\\
0\ar[r] &\mu_N \ar[r] &Tate(q)[N] \ar[r]
&\mathbb{Z}/N\mathbb{Z}\ar[r] &0}$$

Therefore, $Tate(q)[N]$  
is isomorphic to the disjoint union
$$\coprod\limits_{k=0}^{N-1}\mbox{Spec}(\mathbb{Z}((q))[x]/(x^N-q^k)).$$

In addition, we have the question how to classify all the finite
subgroups of $Tate(q)$.  As shown in Proposition 6.5.1,
\cite{KM85}, the ring $O_{Sub_n}$ that classifies subgroups of
$Tate(q)$ of order $n$ exists. To give a description of it, first
we describe the isogenies for the analytic Tate curve over
$\mathbb{C}$.

Let $(d, e)$ be a pair of positive integers such that $N=de$ and
$q'$ a nonzero complex number such that $q^d=q'^e$. The map
\begin{align*}\psi_d: \mathbb{C}^*/q^{\mathbb{Z}}&\longrightarrow \mathbb{C}^*/q'^{\mathbb{Z}}\\x&\mapsto x^d\end{align*}
 is well-defined since $\psi_d(q^{\mathbb{Z}})\subseteq
q'^{\mathbb{Z}}$. The kernel of $\psi_d$ is
$$\{\mu_d^nq^{\frac{m}{e}}q^{\mathbb{Z}}| n, m\in\mathbb{Z}\}$$ where $\mu_d$ is a $d-$th primitive root of 1 and $q^{\frac{1}{e}}$ is a $e-$th primitive root of
$q$. Its order is $N$. In fact
$$\{\mbox{Ker}\psi_d| \mbox{   } d \mbox{  divides }N\mbox{  and
}d\geq 1\}$$ gives all the subgroups of
$\mathbb{C}^*/q^{\mathbb{Z}}$ of order $N$.

\begin{proposition}For each pair of number $(d, e)$, there exists an isogeny
$$\Psi_{d, e}: Tate((q))\longrightarrow Tate((q'))$$ of the elliptic curves over $O_{Sub_n}$ such that
its kernel is the universal subgroup.  \end{proposition}

We have
$$O_{Sub_n}\otimes\mathbb{C}=\prod_{N=de}\mathbb{C}((q))[q']/\langle
q^d-q'^e \rangle.$$ Moreover, we have  the conclusion below.

\begin{proposition} The finite subgroups of the
Tate curve are the kernels of isogenies. \end{proposition}

\subsection{Formulas for Induction} \label{inductionqell} 


Before the main conclusion,  we introduce the induction formula
for quasi-elliptic cohomology. The induction formula for Tate
K-theory is constructed in Section 2.3.3, \cite{Gan13}.

Let $H\subseteq G$ be an inclusion of finite groups and $X$ be a
$G-$space. Then we have the inclusion of the groupoids
$$j:X/\!\!/H\longrightarrow X/\!\!/G.$$

Let $a'=\prod\limits_{\sigma\in H_{conj}}a'_{\sigma}$ be an
element in $QEll_H(X)=\prod\limits_{\sigma\in
H_{conj}}K_{\Lambda_H(\sigma)}(X^{\sigma})$ where $\sigma$ goes
over all the conjugacy classes in $H$. The finite covering map
$$f': \Lambda(G\times_H X/\!\!/G)\longrightarrow  \Lambda(X/\!\!/G)$$  is defined by sending  an object $(\sigma, [g, x])$ to  $(\sigma, gx)$ and
a morphism  $([g', t], (\sigma, [g, x]))$ to $([g', t], (gx,
\sigma))$. The transfer of  quasi-elliptic cohomology
$$\mathcal{I}_H^G: QEll_H(X)\longrightarrow QEll_G(X)$$
is defined to be the composition \begin{equation}
QEll_H(X)\buildrel\cong\over\longrightarrow
QEll_G(G\times_HX)\longrightarrow
QEll_G(X)\label{qelltransfer}\end{equation} where the first map is
the change-of-group isomorphism and the second is the finite
covering.

Thus
$$\mathcal{I}^G_H(a')_{g}=\sum_{r}r\cdot a'_{r^{-1}gr}$$ where $r$ goes
over a set of representatives of $(G/H)^{g}$, in other words,
$r^{-1}gr$ goes over a set of representatives of conjugacy classes
in $H$ conjugate to $g$ in $G$.

\begin{equation}\mathcal{I}^G_H(a')_{g}=\begin{cases}Ind^{\Lambda_G}_{\Lambda_H}(a'_{g}) &\mbox{if  }g\mbox{ is conjuate to some element  }h \mbox{  in
}H;\\ 0 &\mbox{if there is no element conjugate to }g\mbox{  in
}H.
\end{cases}\end{equation}

\bigskip
There is another way to describe the transfer, which is shown in Rezk's unpublished work
\cite{Rez11} for quasi-elliptic cohomology. 
The transfer of Tate K-theory can be described similarly. 

\subsection{The main theorem}\label{proofstrict}

Theorem \ref{stricktheorem} gives a classification of finite
subgroups of the Tate curve and a similar conclusion for the
quasi-elliptic cohomology. We prove it in this section by
representation theory. We assume the readers are familiar with the
transfer ideal $I_{tr}$ of equivariant K-theory. References for
that include Chapter II, \cite{LMSM} and Section 1.8,
\cite{Rez06}.

Let $N$ be an integer. Analogous to the transfer ideal $I_{tr}$ of
equivariant K-theory,  we can define the transfer ideal for Tate
K-theory
\begin{equation}I^{Tate}_{tr}:= \sum_{\substack{i+j=N,\\
N>j>0}}\mbox{Image}[I^{\Sigma_N}_{\Sigma_i\times\Sigma_j}:
K_{Tate}(\mbox{pt}/\!\!/\Sigma_i\times\Sigma_j)\longrightarrow
K_{Tate}(\mbox{pt}/\!\!/\Sigma_N)]\label{transferidealtatek}\end{equation}
where $I^G_H$ is the transfer map of $K_{Tate}$ along
$H\hookrightarrow G$ defined in Proposition 2.23, \cite{Gan13},
and the transfer ideal for quasi-elliptic cohomology
\begin{equation}\mathcal{I}^{QEll}_{tr}:= \sum_{\substack{i+j=N,\\
N>j>0}}\mbox{Image}[\mathcal{I}^{\Sigma_N}_{\Sigma_i\times\Sigma_j}:
QEll(\mbox{pt}/\!\!/\Sigma_i\times\Sigma_j)\longrightarrow
QEll(\mbox{pt}/\!\!/\Sigma_N)]\label{transferidealqec}\end{equation}
with $\mathcal{I}^G_H$ the transfer map of $QEll$ along
$H\hookrightarrow G$ defined in (\ref{qelltransfer}).

\begin{theorem}\label{stricktheorem} The Tate K-theory of symmetric groups modulo the transfer
ideal $I^{Tate}_{tr}$ classifies the finite subgroups of the Tate
curve. Explicitly,
\begin{equation}(K_{Tate})_{\Sigma_N}(\mbox{pt})/I^{Tate}_{tr}\cong
\prod_{N=de}\mathbb{Z}((q))[q']/\langle q^d-q'^e
\rangle,\label{tatec}\end{equation} where  $q'$ is the image of
$q$ under the power operation $P^{Tate}$ constructed in Definition
3.15, \cite{Gan13}. The product goes over all the ordered pairs of
positive integers $(d, e)$ such that $N=de$.

\bigskip

We have the analogous conclusion for quasi-elliptic cohomology.
\begin{equation}QEll_{\Sigma_N}(\mbox{pt})/\mathcal{I}^{QEll}_{tr}\cong
\prod_{N=de}\mathbb{Z}[q^{\pm}][q']/\langle q^d-q'^e
\rangle,\label{ellc}\end{equation} where $q'$ is the image of $q$
under the power operation $\mathbb{P}_N$ constructed in Section
\ref{s2complete}. The product goes over all the ordered pairs of
positive integers $(d, e)$ such that $N=de$.

\end{theorem}

We show the proof of (\ref{ellc}). The proof of (\ref{tatec}) is
similar.

\begin{proof}[Proof of (\ref{ellc})]

We divide the elements in $\Sigma_N$ into two cases.

\textbf{Case I:}

The decomposition of $\sigma$ has cycles of different length. For
example, the element $$(1\mbox{  } 2)(3\mbox{  } 4)(5\mbox{  }
6)(7 \mbox{ }8\mbox{  } 9 \mbox{ }10)(11\mbox{ } 12\mbox{  }
13\mbox{ } 14)(15 \mbox{ }16\mbox{ } 17)\in\Sigma_{17}$$ is in
this case and $(1\mbox{   }2)(3\mbox{   }4)(5\mbox{   }6)$,
$(1\mbox{ }2\mbox{   }3\mbox{   }4\mbox{   }5)(6\mbox{   }7\mbox{
}8\mbox{ }9\mbox{   }10)$ are not.

Most elements in $\Sigma_N$ belong to Case I. $\sigma$ is not in
this case if and only if it consists of cycles of the same length,
such as $(1\mbox{  }2)(3\mbox{  }4)$, $(1\mbox{  }2\mbox{
}3)$, $1$, $(1\mbox{  }2\mbox{  }3)(4\mbox{  }5\mbox{  }6)$.

For those $\sigma$ that belong to Case I,
$\Lambda_{\Sigma_N}(\sigma)=\Lambda_{\Sigma_r\times\Sigma_{N-r}}(\sigma)$,
so
$Ind^{\Lambda_{\Sigma_N}(\sigma)}_{\Lambda_{\Sigma_r\times\Sigma_{N-r}}(\sigma)}$
is the identity map, so
$K_{\Lambda_{\Sigma_N}(\sigma)}(\mbox{pt})$ is equal to
$Ind^{\Lambda_{\Sigma_N}(\sigma)}_{\Lambda_{\Sigma_r\times\Sigma_{N-r}}(\sigma)}
K_{\Lambda_{\Sigma_r\times\Sigma_{N-r}}(\sigma)}(\mbox{pt})$.
Thus, the summand corresponding to $\sigma$ in
$QEll(\mbox{pt}/\!\!/\Sigma_N)$ is completely cancelled. 

\bigskip

\textbf{Case II:}

$\sigma$ consists of cycles of the same length. In other words, it
consists of $d$ $e-$cycles with $N=de$. 

The centralizer $C_{\Sigma_N}(\sigma)\cong C_e\wr \Sigma_d$, where
$C_e=\mathbb{Z}/e\mathbb{Z}$ is the cyclic group with order $e$.            
We have $$\Lambda_{\Sigma_N}(\sigma)\cong
\Lambda_{\Sigma_e}(12\cdots e)\wr_{\mathbb{T}}\Sigma_d$$ is the
subgroup of $\Lambda_{\Sigma_e}(12\cdots e)\wr\Sigma_d$ with
elements of the form $$([a_1,t], [a_2, t], \cdots [a_d, t];
\tau),\mbox{   with    }a_1, \cdots a_d\in C_e,\mbox{   } \tau\in
\Sigma_d, \mbox{  }t\in\mathbb{R}.$$
$K_{\Lambda_{\Sigma_N}(\sigma)}(\mbox{pt})$ is the representation
ring $R\Lambda_{\Sigma_N}(\sigma)$. According to Theorem
\ref{repfibwr}, as a $\mathbb{Z}[q^{\pm}]-$module, it has the
basis
\begin{align*}\{&Ind^{\Lambda_{\Sigma_e}(12\cdots
e)\wr_{\mathbb{T}}\Sigma_d}_{\Lambda_{\Sigma_e}(12\cdots
e)\wr_{\mathbb{T}}\Sigma_{(d)}}(q^{\frac{a_1}{e}})^{\otimes_{\mathbb{Z}[q^{\pm}]}d_1}\otimes_{\mathbb{Z}[q^{\pm}]}\cdots\otimes_{\mathbb{Z}[q^{\pm}]}(q^{\frac{a_r}{e}})^{\otimes_{\mathbb{Z}[q^{\pm}]}d_r}\otimes
D_{\tau}\mbox{ }|\\&(d)=(d_1, d_2, \cdots d_r)\mbox{ is a
partition of }d. \\ &a_1, a_2, \cdots a_r\mbox{ are in
 }\{0, 1, \cdots e-1\}. \mbox{     }\tau\in R\Sigma_{(d)}\mbox{  is  irreducible}.\}\end{align*}
where for each $a\in\mathbb{Z}$, $q^{\frac{a}{e}}:
\Lambda_{C_e}((12\cdots e))\longrightarrow U(1)$ is the map
\begin{equation}q^{\frac{a}{e}}([(12\cdots e)^j, t])=e^{2\pi
ia\frac{j+t}{e}}.\end{equation} Namely, it is the map $x_1^a$ in
the sense of Example \ref{ppex}. 

For each partition $(d)$ of $d$, if it has more than one cycle, 
$\Sigma_{(d)}$ is a subgroup of some $\Sigma_{d_1}\times
\Sigma_{d-d_1}$ for some positive integer $0<d_1<d$. So for each
$$Ind^{\Lambda_{\Sigma_e}(12\cdots
e)\wr_{\mathbb{T}}\Sigma_d}_{\Lambda_{\Sigma_e}(12\cdots
e)\wr_{\mathbb{T}}\Sigma_{(d)}}(q^{\frac{a_1}{e}})^{\otimes_{\mathbb{Z}[q^{\pm}]}d_1}\otimes_{\mathbb{Z}[q^{\pm}]}\cdots\otimes_{\mathbb{Z}[q^{\pm}]}(q^{\frac{a_r}{e}})^{\otimes_{\mathbb{Z}[q^{\pm}]}d_r}\otimes
D_{\tau}$$ with $r\geq 2$, it is equal to
\begin{align*}Ind^{\Lambda_{\Sigma_e}(12\cdots
e)\wr_{\mathbb{T}}\Sigma_d}_{\Lambda_{\Sigma_e}(12\cdots
e)\wr_{\mathbb{T}}(\Sigma_{d_1}\times\Sigma_{d-d_1})}&(Ind^{\Lambda_{\Sigma_e}(12\cdots
e)\wr_{\mathbb{T}}(\Sigma_{d_1}\times\Sigma_{d-d_1})}_{\Lambda_{\Sigma_e}(12\cdots
e)\wr_{\mathbb{T}}\Sigma_{(d)}}(q^{\frac{a_1}{e}})^{\otimes_{\mathbb{Z}[q^{\pm}]}d_1}\otimes_{\mathbb{Z}[q^{\pm}]}\\
&\cdots\otimes_{\mathbb{Z}[q^{\pm}]}(q^{\frac{a_r}{e}})^{\otimes_{\mathbb{Z}[q^{\pm}]}d_r}\otimes
D_{\tau})\end{align*} by the property of induced representation.
Note that $$\Lambda_{\Sigma_e}(12\cdots
e)\wr_{\mathbb{T}}(\Sigma_{d_1}\times\Sigma_{d-d_1})\cong
\Lambda_{\Sigma_{d_1e}\times\Sigma_{N-d_1e}}(\sigma).$$ So
$$Ind^{\Lambda_{\Sigma_e}(12\cdots
e)\wr_{\mathbb{T}}(\Sigma_{d_1}\times\Sigma_{d-d_1})}_{\Lambda_{\Sigma_e}(12\cdots
e)\wr_{\mathbb{T}}\Sigma_{(d)}}(q^{\frac{a_1}{e}})^{\otimes_{\mathbb{Z}[q^{\pm}]}d_1}\otimes_{\mathbb{Z}[q^{\pm}]}\cdots\otimes_{\mathbb{Z}[q^{\pm}]}(q^{\frac{a_r}{e}})^{\otimes_{\mathbb{Z}[q^{\pm}]}d_r}\otimes
D_{\tau}$$ is in
$K_{\Lambda_{\Sigma_{d_1e}\times\Sigma_{N-d_1e}}(\sigma)}(\mbox{pt})$,
Thus, each
base element with $r\geq 2$ is contained in the transfer ideal. 

When $r=1$, consider $$(q^{\frac{a_1}{e}})^{\otimes_{\mathbb{Z}[q^{\pm}]}d}\otimes D_{\tau}$$ with $\tau\in R\Sigma_d$. 
As indicated in Proposition 1.1 and Corollary 1.5 in \cite{Ati66},
each $\tau$, except the trivial representation of $\Sigma_d$, can
be induced from a representation $\tau'$ in some
$R(\Sigma_i\times\Sigma_{d-i})$ with $d>i>0$. 

 \textbf{Claim: }the representation $$Ind^{\Lambda_{\Sigma_e}(12\cdots
e)\wr_{\mathbb{T}}\Sigma_d}_{\Lambda_{\Sigma_{e}}(12\cdots
e)\wr_{\mathbb{T}}(\Sigma_i\times\Sigma_{d-i})}
(q^{\frac{a_1}{e}})^{\otimes_{\mathbb{Z}[q^{\pm}]}i}\otimes_{\mathbb{Z}[q^{\pm}]}
(q^{\frac{a_1}{e}})^{\otimes_{\mathbb{Z}[q^{\pm}]}(d-i)} \otimes
D_{\tau'}$$ is isomorphic to
$$(q^{\frac{a_1}{e}})^{\otimes_{\mathbb{Z}[q^{\pm}]}d}\otimes
D_{Ind^{\Sigma_d}_{\Sigma_i\times\Sigma_{d-i}}\tau'},$$ which is
$$(q^{\frac{a_1}{e}})^{\otimes_{\mathbb{Z}[q^{\pm}]}d}\otimes D_{\tau}.$$

To prove this, we consider a set
$\{\tau_{\alpha}\}_{\alpha\in\Sigma_d/\Sigma_i\times\Sigma_{d-i}}$
of coset representatives. Then $$\{\eta_{\alpha}:=(1, \cdots 1;
\tau_{\alpha})\}_{\alpha\in\Sigma_d/\Sigma_i\times\Sigma_{d-i}}$$
is a set of coset representatives of
$$\big(\Lambda_{\Sigma_e}(12\cdots e)\wr_{\mathbb{T}}\Sigma_d  \big) /\big( \Lambda_{\Sigma_{e}}(12\cdots e)\wr_{\mathbb{T}}(\Sigma_i\times\Sigma_{d-i})\big).$$

Let $W$ be a representation space of
$\Lambda_{\Sigma_{e}}(12\cdots
e)\wr_{\mathbb{T}}(\Sigma_i\times\Sigma_{d-i})$, Then
$$Ind^{\Lambda_{\Sigma_e}(12\cdots
e)\wr_{\mathbb{T}}\Sigma_d}_{\Lambda_{\Sigma_{e}}(12\cdots
e)\wr_{\mathbb{T}}(\Sigma_i\times\Sigma_{d-i})}W$$ is the direct
product of $[\Sigma_d :\Sigma_i\times\Sigma_{d-i}]$ copies of $W$.
For any element $$H=(g_1, \cdots g_d; \beta)\in
\Lambda_{\Sigma_e}(12\cdots e)\wr_{\mathbb{T}}\Sigma_d,$$ and each
$\alpha\in\Sigma_d/\Sigma_i\times\Sigma_{d-i}$, there is a unique
$\alpha'\in\Sigma_d/\Sigma_i\times\Sigma_{d-i}$ and a unique $$
J_{\alpha}=(g'_1, \cdots g'_d; \gamma_{\alpha})\in
\Lambda_{\Sigma_{e}}(12\cdots
e)\wr_{\mathbb{T}}(\Sigma_i\times\Sigma_{d-i})$$ such that
$H\eta_{\alpha}=\eta_{\alpha'}J_{\alpha}$. Note that $$g'_1,
\cdots g'_d$$ is a permutation of $$g_1, \cdots g_d.$$ So
$(q^{\frac{a_1}{e}})^{\otimes_{\mathbb{Z}[q^{\pm}]}d}(g'_1, \cdots
g'_d)= (q^{\frac{a_1}{e}})^{\otimes_{\mathbb{Z}[q^{\pm}]}d}(g_1,
\cdots g_d).$ In addition,
$\beta\tau_{\alpha}=\tau_{\alpha'}\gamma_{\alpha}$. Let
$$\prod_{\alpha}w_{\alpha}$$ be an element in
$$Ind^{\Lambda_{\Sigma_e}(12\cdots
e)\wr_{\mathbb{T}}\Sigma_d}_{\Lambda_{\Sigma_{e}}(12\cdots
e)\wr_{\mathbb{T}}(\Sigma_i\times\Sigma_{d-i})}W.$$ We have
\begin{align*}&\bigg(Ind^{\Lambda_{\Sigma_e}(12\cdots
e)\wr_{\mathbb{T}}\Sigma_d}_{\Lambda_{\Sigma_{e}}(12\cdots
e)\wr_{\mathbb{T}}(\Sigma_i\times\Sigma_{d-i})}
(q^{\frac{a_1}{e}})^{\otimes_{\mathbb{Z}[q^{\pm}]}i}\otimes_{\mathbb{Z}[q^{\pm}]}
(q^{\frac{a_1}{e}})^{\otimes_{\mathbb{Z}[q^{\pm}]}d-i}
\otimes D_{\tau'}\bigg)(H)(\prod_{\alpha}w_{\alpha})\\
&=\prod_{\alpha}J_{\alpha}w_{\beta(\alpha)}
=\prod_{\alpha}(q^{\frac{a_1}{e}})^{\otimes_{\mathbb{Z}[q^{\pm}]}d}(g_1, \cdots g_d)D_{\tau'}(1,\cdots 1;\gamma_{\alpha})(w_{\beta\alpha})\\
&=(q^{\frac{a_1}{e}})^{\otimes_{\mathbb{Z}[q^{\pm}]}d}(g_1, \cdots g_d)\prod_{\alpha}\tau'(\gamma_{\alpha})(w_{\beta\alpha})\\
&=(q^{\frac{a_1}{e}})^{\otimes_{\mathbb{Z}[q^{\pm}]}d}(g_1, \cdots g_d)(Ind^{\Sigma_d}_{\Sigma_i\times\Sigma_{d-i}}\tau')(\beta)(\prod_{\alpha}w_{\alpha})\\
&=(q^{\frac{a_1}{e}})^{\otimes_{\mathbb{Z}[q^{\pm}]}d}(g_1, \cdots g_d; \beta)D_{Ind^{\Sigma_d}_{\Sigma_i\times\Sigma_{d-i}}\tau'}(g_1, \cdots g_d; \beta)(\prod_{\alpha}w_{\alpha})\\
&= \big((q^{\frac{a_1}{e}})^{\otimes_{\mathbb{Z}[
q^{\pm}]}d}\otimes
D_{Ind^{\Sigma_d}_{\Sigma_i\times\Sigma_{d-i}}\tau'}\big)(g_1,
\cdots g_d; \beta)(\prod_{\alpha}w_{\alpha})\end{align*}

So the claim is proved.
\bigskip

Since $$\{Ind^{\Sigma_d}_{\Sigma_i\times\Sigma_{d-i}}\tau'\mbox{ }
|   \mbox{   }\tau'\in R(\Sigma_i\times\Sigma_{d-i})\mbox{ and }
i=1, 2, \cdots d-1.\}$$ contains all the irreducible
representation of $\Sigma_d$ except the trivial representation,
which is corresponding to the partition $(d)$,  thus, by the
claim, $K_{\Lambda_{\Sigma_N}(\sigma)}(\mbox{pt})$ modulo the
image of the transfer, is a $\mathbb{Z}[q^{\pm}]-$module generated
by the equivalent classes represented by
\begin{equation}\{((q^{\frac{a}{e}})^{\otimes_{\mathbb{Z}[q^{\pm}]}d})^{\sim}\mbox{  }|\mbox{  } a=0, 1, \cdots
e-1\}.\label{get}\end{equation}

For any $a$, $(q^{\frac{a}{e}})^{\otimes_{\mathbb{Z}[q^{\pm}]}d}$
is $(q^{\frac{1}{e}})^{\otimes_{\mathbb{Z}[q^{\pm}]}d}$ to the
$a-$th power. Note that, by (\ref{ee}),
$(q^{\frac{1}{e}})^{\otimes_{\mathbb{Z}[q^{\pm}]}d}$ is
$$q':=\mathbb{P}_{\sigma}(q).$$


\bigskip

To get the isomorphism (\ref{ellc}), consider a map $$\Psi:
\mathbb{Z}[q^{\pm}][x]\longrightarrow
K_{\Lambda_{\Sigma_N}(\sigma)}(\mbox{pt})/\mathcal{I}^{QEll}_{tr}$$
by sending $q$ to $q$ and $x$ to $q'$, which is a well-defined
$\mathbb{Z}[q^{\pm}]-$homomorphism.

Since $q'^e= q^d$,
$K_{\Lambda_{\Sigma_N}(\sigma)}(\mbox{pt})/\mathcal{I}^{QEll}_{tr}$
is a $\mathbb{Z}[q^{\pm}]-$module generated by $$1, q', \cdots
q'^{e-1}.$$ So any element in it can be expressed as
$$\sum_{j=0}^{e-1}f_j(q)q'^j$$ where each $f_j(q)$ is in the
polynomial ring $\mathbb{Z}[q^{\pm}]$. It is the image of
$$\sum_{j=0}^{e-1}f_j(q)x^j$$ in $\mathbb{Z}[q^{\pm}][x]$. So $\Psi$ is surjective.

Then we study its kernel. If $$F:=\sum_{j=0}^{e-1}f_j(q)q'^j$$
is in $\mathcal{I}^{QEll}_{tr}$, then it is in $\mathbb{Z}[q^{\pm}]$. 
So we can assume $F=0$.


For each element $[(a_1, \cdots a_d; \beta), t]$ in
$\Lambda_{\Sigma_N}(\sigma)$ with $(a_1, \cdots a_d; \beta)\in
C_{\Sigma_N}(\sigma),$
\begin{equation}q([(a_1, \cdots a_d;
\beta), t])=e^{2\pi it},\end{equation} 
\begin{equation}q'([(a_1, \cdots a_d; \beta),
t])=e^{\frac{2\pi i(a_1+\cdots a_d+dt)}{e}}.\end{equation}

\begin{align*}F([(a_1, \cdots a_d; \beta),
t])&=\sum_{j=0}^{e-1}f_j(q)q'^j([(a_1, \cdots a_d; \beta),
t])\\=\sum_{j=0}^{e-1}f_j(e^{2\pi it})e^{\frac{2\pi ij(a_1+\cdots
+a_d+dt)}{e}} &=\sum_{j=0}^{e-1}f_j(e^{2\pi it})e^{\frac{2\pi
ijdt}{e}}e^{\frac{2\pi ij(a_1+\cdots +a_d)}{e}}.\end{align*}

Let $$F_j(t): =f_j(e^{2\pi it})e^{\frac{2\pi ijdt}{e}}$$ be the
complex-valued function in the variable $t$. Let $\alpha$ denote
the number $e^{\frac{2\pi i}{e}}$. The integers
$$(a_1+\cdots +a_d)$$ go over $0, 1, \cdots e-1$. Consider the
$e$ equations
$$\sum_{j=0}^{e-1}F_j(t)\alpha^{jk}=0\mbox{, for }k=0, 1, \cdots e-1.$$

In other words, \[ \left( \begin{array}{ccccc}
1 & 1 & 1 &\cdots &1 \\
1 & \alpha & \alpha^2 &\cdots &\alpha^{e-1} \\
1 & \alpha^2 & \alpha^4 &\cdots &\alpha^{2(e-1)}\\
\vdots&\vdots&\vdots&&\vdots\\
1 &\alpha^{e-1} &\alpha^{2(e-1)} &\cdots &\alpha^{(e-1)^2}
\end{array} \right)\left(
\begin{array}{c}
F_0(t) \\
F_1(t) \\
F_2(t) \\
\vdots \\
F_{e-1}(t)\end{array} \right)=0\]

The determinant of the Vandermonde matrix \[\left(
\begin{array}{ccccc}
1 & 1 & 1 &\cdots &1 \\
1 & \alpha & \alpha^2 &\cdots &\alpha^{e-1} \\
1 & \alpha^2 & \alpha^4 &\cdots &\alpha^{2(e-1)}\\
\vdots&\vdots&\vdots&&\vdots\\
1 &\alpha^{e-1} &\alpha^{2(e-1)} &\cdots &\alpha^{(e-1)^2}
\end{array} \right)\] is \begin{equation}\prod_{j=0}^{e-2}\prod_{k=j+1}^{e-1}(\alpha^{k}-\alpha^{j}).\label{det}\end{equation}


When $\alpha =e^{\frac{2\pi i}{e}}$, each
$(\alpha^{k}-\alpha^{j})$ in the product (\ref{det}) is nonzero,
so for any $e$, the determinant is nonzero and the matrix is
non-singular. So we get $F_j(t)=0$ for any $t\in\mathbb{R}$ and
$j=0, 1, 2, \cdots e-1$.

So each $f_j(q)$ in $F$ is the zero polynomial.

The kernel of $\Psi$ is the ideal generated by $q'^e-q^d$.

\end{proof}

From the power operation of quasi-elliptic cohomology, we can
construct a new operation for quasi-elliptic cohomology.
\begin{proposition}The composition
\begin{align*}\overline{P}_N: &QEll_{G}(X) \buildrel{\mathbb{P}_N}\over\longrightarrow QEll_{G\wr\Sigma_N}(X^{\times N})
\buildrel{res}\over\longrightarrow QEll_{G\times
\Sigma_N}(X^{\times N})\\  &\buildrel{diag^*}\over\longrightarrow
 QEll_{G\times\Sigma_N}(X) \cong QEll_G(X)\otimes_{\mathbb{Z}[q^{\pm}]} QEll_{\Sigma_N}(\mbox{pt}) \\
  &\longrightarrow QEll_G(X)\otimes_{\mathbb{Z}[q^{\pm}]} QEll_{\Sigma_N}(\mbox{pt})/\mathcal{I}^{QEll}_{tr} \\ &\cong
 QEll_G(X)\otimes_{\mathbb{Z}[q^{\pm}]}\prod_{N=de}\mathbb{Z}[q^{\pm}][q']/\langle q^d-q'^e
\rangle\end{align*} defines a ring homomorphism, where $res$ is
the restriction map by the inclusion
$$G\times\Sigma_N\hookrightarrow G\wr\Sigma_N\mbox{,    } (g, \sigma)\mapsto (g, \cdots g;
\sigma),$$ $diag$ is the diagonal map $$X\longrightarrow X^{\times
N}\mbox{,    } x\mapsto (x, \cdots x)$$ and the last map is the
isomorphism (\ref{ellc}).\label{adamsqell}
\end{proposition}

\begin{proof}

Let $V=\prod\limits_{g\in G_{conj}}V_{g}\in QEll_G(X)$. Apply the
explicit formula of the power operation in (\ref{ee}), the
composition $diag^*\circ res\circ\mathbb{P}_N$ sends $V$ to
$$\prod_{\substack{g\in G_{conj} \\ \sigma\in
{\Sigma_N}_{conj}}} \otimes_k\otimes_{(i_1, \cdots
i_k)}V_{g^k}q^{\frac{1}{k}}$$ where $(i_1, \cdots i_k)$ goes over
all the $k-$cycles of $\sigma$, and the tensor products are those
of the $\mathbb{Z}[q^{\pm}]-$algebras. Then, as shown in the proof
of (\ref{ellc}), after taking the quotient by the transfer ideal
$\mathcal{I}^{QEll}_{tr}$, all the factors in $diag^*\circ
res\circ\mathbb{P}_N(V)$ are cancelled except those corresponding
to the elements in ${\Sigma_N}_{conj}$ with cycles of the same
length.  For the factor corresponding to the element $\sigma\in
{\Sigma_N}_{conj}$ with $d$ $e-$cycles and $de=N$, the nontrivial
part is $V_{g^e, d}\otimes_{\mathbb{Z}[q^{\pm}]} q'_{d, e}$ where
$V_{g^e, d}$ is the fixed point space of
$V_{g^e}^{\otimes_{\mathbb{Z}[q, q^{-1}]}d}$ by the permutations
$\Sigma_d$ and $q'_{d, e}= \mathbb{P}_{\sigma}(q) =
(q^{\frac{1}{e}})^{\otimes_{\mathbb{Z}[q, q^{-1}]}d}$.

Thus, \begin{equation}\overline{P}_N(V)= \prod\limits_{\substack{g\in G_{conj} \\
N=de}}V_{g^e, d}\otimes_{\mathbb{Z}[q^{\pm}]} q'_{d,
e}.\end{equation}

Let $V, W$ be two elements in $QEll_G(X)$. We have $$(V\oplus
W)_{g^e, d}=V_{g^e, d}\oplus W_{g^e, d}\mbox{   and    }(V\otimes
W)_{g^e, d}=V_{g^e, d}\otimes W_{g^e, d}.$$
\begin{align*}\overline{P}_N(V\oplus W)&=\prod\limits_{\substack{g\in G_{conj} \\ N=de}}
(V\oplus W)_{g^e, d}\otimes_{\mathbb{Z}[q^{\pm}]}
q'_{d, e}\\ 
&=\bigg( \prod\limits_{\substack{g\in G_{conj} \\
N=de}}V_{g^e, d}\otimes_{\mathbb{Z}[q^{\pm}]} q'_{d, e}\bigg)
\oplus
\bigg(\prod\limits_{\substack{g\in G_{conj} \\
N=de}} W_{g^e, d}\otimes_{\mathbb{Z}[q^{\pm}]} q'_{d, e}\bigg)
\\ &=\overline{P}_N(V)\oplus\overline{P}_N(W) .\end{align*}
Similarly, \begin{align*}\overline{P}_N(V\otimes
W)&=\prod\limits_{\substack{g\in G_{conj} \\ N=de}} (V\otimes
W)_{g^e, d}\otimes_{\mathbb{Z}[q^{\pm}]}
q'_{d, e}\\ 
&=\bigg( \prod\limits_{\substack{g\in G_{conj} \\
N=de}}V_{g^e, d}\otimes_{\mathbb{Z}[q^{\pm}]} q'_{d, e}\bigg)
\otimes
\bigg(\prod\limits_{\substack{g\in G_{conj} \\
N=de}} W_{g^e, d}\otimes_{\mathbb{Z}[q^{\pm}]} q'_{d, e}\bigg)
\\ &=\overline{P}_N(V)\otimes\overline{P}_N(W) .\end{align*}
\end{proof}


\bibliographystyle{amsplain}


\end{document}